\documentclass[10pt,a4paper]{article}
\usepackage[utf8]{inputenc}
\usepackage{amsmath}
\usepackage{amsfonts}
\usepackage{amssymb}
\usepackage{amsthm}
\usepackage{enumitem}
\usepackage{color}
\usepackage{hyperref}
\usepackage{MnSymbol}
\usepackage{graphicx}

\usepackage[a4paper, textwidth=13cm]{geometry}

\newcommand{\R}{\mathbb{R}}
\newcommand{\Z}{\mathbb{Z}}
\newcommand{\N}{\mathbb{N}}

\newcommand{\geps}{\Gamma_{\epsilon}}

\newcommand{\veps}{v_{\epsilon}}

\newcommand{\ueps}{u_{\epsilon}}

\newcommand{\oem}{\Omega_{\epsilon}^M}

\newcommand{\U}{\mathcal{U}}

\newcommand{\tueps}{\tilde{u}_{\epsilon}}

\newcommand{\sepm}{S_{\epsilon}^{\pm}}
\newcommand{\x}{\bar{x}}
\newcommand{\fxe}{\dfrac{x}{\epsilon}}
\newcommand{\hZ}{\widehat{Z}}
\newcommand{\E}{\mathcal{E}}
\newcommand{\e}{\mathfrak{e}}

\newcommand{\oems}{\Omega_{\epsilon}^{M,s}}

\newcommand{\foe}{\dfrac{1}{\epsilon}}

\newcommand{\bxfxe}{\left(\bar{x},\dfrac{x}{\epsilon}\right)}
\newcommand{\reps}{r^{\epsilon}}
\newcommand{\treps}{\tilde{r}^{\epsilon}}

\newcommand{\phieps}{\phi_{\epsilon}}

\newcommand{\rats}{\overset{t.s.}{\longrightarrow}}

\newcommand{\y}{\bar{y}}
\newcommand{\tf}{\tilde{f}}
\newcommand{\hueps}{\widehat{u}_{\epsilon}}

\newcommand{\bphi}{\bar{\phi}}
\newcommand{\Reps}{\mathcal{R}_{\epsilon}}
\renewcommand{\reps}{r_{\epsilon}}

\renewcommand{\div}{\mathrm{div}}

\newcommand{\gr}{>}
\newcommand{\kl}{<}

\newtheorem{definition}{Definition}
\newtheorem{remark}{Remark}
\newtheorem{theorem}{Theorem}
\newtheorem{proposition}{Proposition}
\newtheorem{lemma}{Lemma}
\newtheorem{corollary}{Corollary}

\title{Two-scale tools for homogenization and dimension reduction of perforated thin layers:  Extensions, Korn-inequalities, and two-scale compactness of scale-dependent sets in Sobolev spaces}
\date{}
\author{M. Gahn$^{*}$, W. J\"ager$^{*}$, M. Neuss-Radu$^{**}$}

\begin{document}

\maketitle

\begin{abstract}
This investigation develops basic methods  for the multi-scale analysis for problems in  thin porous layers. More precisely, we provide tools for the homogenization in case of "tangentially” periodic structures and dimensional reduction letting the layer thickness tend to zero proportional to the scale parameter $\epsilon$. A crucial point is the identification of
 scale limits of functions $\veps$ in subsets of function spaces characterized by uniform \textit{a priori} estimates with respect to $\epsilon$, arising for solutions of differential equations in heterogeneous media with thin layers, e.\,g. of a Navier-Stokes system, models in linear elasticity, or problems with fluid-structure interaction. Often in problems from continuum mechanics, 
 in a first step, the symmetric gradients of arising vector fields can be controlled and Korn´s inequality in porous layers is required  to estimate the gradients, such that crucial constants do not depend on $\epsilon$.
Controllable pore-filling extension are constructed and, thus, the analysis is reduced to a fixed basic domain. 
The proof of the required Korn-inequalities for porous thin layers, formulated with respect to $L^p$-spaces, is based on these constructions. Also, the investigation of compactness with respect to two-scale convergence  and the characterization of the scale limits is strongly based on the extension theorem and the Korn-inequalities.
To illustrate the range of applications of the developed analytic multiscale method a semi-linear elastic wave equation in a thin periodically perforated layer with an inhomogeneous Neumann boundary condition on the surface of the elastic substructure is treated and an homogenized, reduced system is derived.

\end{abstract}

\let\thefootnote\relax\footnotetext{$^{*}$Interdisciplinary Center for Scientific Computing, University of Heidelberg, Im Neuenheimer Feld
	205, 69120 Heidelberg, Germany, markus.gahn@iwr.uni-heidelberg.de; wjaeger@iwr.uni-heidelberg.de.
	
	\vspace{.5mm}
	$^{**}$Department Mathematik, Friedrich-Alexander-Universität Erlangen-Nürnberg, Cauerstr. 11, 91058 Erlangen, Germany, maria.neuss-radu@math.fau.de }

\section{Introduction}

In this paper we develop general multiscale methods for homogenization and dimension reduction for thin perforated layers, which are in particular important for problems in continuum mechanics like linear elasticity and fluid dynamics. The periodicity and the thickness of the layer is of order $\epsilon$, where the parameter $\epsilon$ is small compared to the length of the layer.
We have to cope with a simultaneous scale transition for the homogenization of the "tangentially" periodic structure and the dimension reduction when the thickness of the layer reduces to zero. For this we derive two-scale compactness results for $\epsilon$-dependent sets in Sobolev spaces. These compactness results depend on uniform \textit{a priori} estimates with respect to $\epsilon$ in Sobolev norms. However, for problems in continuum mechanics, usually  only the symmetric gradient of arising vector fields can be controlled in a first step. To obtain estimates for the $W^{1,p}$-norms, we derive  Korn-inequalities for porous thin layers with explicit dependence on the scaling parameter $\epsilon$.  For the proof we construct an extension operator for Sobolev functions from the perforated layer to the whole layer, where the extension, and especially its symmetric gradient, can be controlled sufficiently well with respect to $\epsilon$. Such extension operators are of crucial importance in homogenization theory, because they allow to reduce the analysis to fixed basis domains (independent of the scaling parameter). Especially, our two-scale compactness results and the characterization of the scale limits is strongly based on the extension operator. Further, they are important for the treatment of nonlinear problems, since they allow the application of compact embeddings  from functional analysis.
\\
In the homogenization theory it is well-known, see \cite{Acerbi1992,CioranescuSJPaulin}, that for a  periodically perforated connected Lipschitz-domain $\Omega_{\epsilon}$, every function $\ueps \in W^{1,p}(\Omega_{\epsilon})$ can be extended to a function $\tueps \in W^{1,p}(\Omega)$, such that
\begin{align*}
\Vert \tueps \Vert_{L^p(\Omega)} \le C \Vert \ueps \Vert_{L^p(\Omega_{\epsilon})}, \qquad \Vert \nabla \tueps \Vert_{L^p(\Omega)} \le C \Vert \nabla \ueps \Vert_{L^p(\Omega_{\epsilon})},
\end{align*}
with a constant $C\gr 0$ independent of $\epsilon$. 
In problems in continuum mechanics, like linear elasticity or the Navier-Stokes-equations, we usually obtain a bound for the symmetric gradient. Hence, the Korn-inequality guarantees the boundedness of the $W^{1,p}$-norm of the function. However, a critical question is in which way the constant in the Korn-inequality depends on $\epsilon$. In this paper we construct an extension operator for $W^{1,p}$-functions defined in a perforated thin layer to the whole layer, such that the symmetric gradient of the extension can be controlled by the symmetric gradient of the function itself uniformly with respect to $\epsilon$. More precisely, for the perforated layer $\oems$ and for  fields $\ueps \in W^{1,p}(\oems)^n$ we construct an extension $\hueps \in W^{1,p}(\oem)^n$ to the whole (thin) layer $\oem$, which in particular fulfills
\begin{align*}
\Vert D(\hueps) \Vert_{L^p(\oem)} \le C \Vert D(\ueps)\Vert_{L^p(\oems)},
\end{align*}
where $D$ denotes the symmetric gradient. Now, using a simple rescaling argument for the thin layer $\oem$, the extension $\hueps$ can be treated as a vector field on a fixed domain with standard embedding results. The key idea of the proof is to use a decomposition argument for the thin perforated layer, see \cite{Acerbi1992} for periodically perforated domains, and a local extension operator preserving the bound for the symmetric gradient. The latter one is obtained by applying the local extension operator from \cite{Acerbi1992} to a vector field modified with a rigid-displacement, which allows to use a Korn-inequality. The extension operator for the perforated layer is the basis for the derivation of Korn-inequalities in the thin porous layer and the two-scale compactness results.
\\
While there exists a huge literature on Korn-inequalities on (fixed) domains \cite{ciarlet1988mathematical,duvant2012inequalities,Necas,Oleinik1992},   and also thin layers \cite{chechkin2000weighted,ciarlet1997mathematical,lewicka2011uniform,zhikov2004korn}, such results are rare for perforated thin layers. Here we have to mention the work \cite{griso2020homogenization} where a Korn-inequality for a thin porous layer (not necessarily with Lipschitz-boundary) similar to our result is shown by using a decomposition of the layer with an interpolation argument applied to a decomposition for a vector field introduced in \cite{griso2008decompositions}. However, the situation seems to be slightly different compared to our setting, since in \cite{griso2020homogenization} the Korn-inequality is proved for functions vanishing in a neighborhood of the lateral boundary of $\partial \oems$ and therefore the constructed interpolation vanishes on the entire lateral boundary of the whole layer $\oem$. Since in our case  we consider functions vanishing on the lateral boundary, but not in its neighborhood, we cannot guarantee that the extension is zero on the lateral boundary of the whole layer.
\\
There is a huge literature on the derivation of plate and shell equations, see for example \cite{ciarlet1997mathematical,ciarlet2000theory},  which is usually based on a rescaling argument for the thin layer leading to function spaces on scale independent domains. However, since we are working on $\epsilon$-depending Sobolev spaces on heterogeneous domains,  convergence concepts adapted to this situation are necessary. We use the method of two-scale convergence for thin perforated layers. This method was investigated for periodic domains in \cite{Allaire_TwoScaleKonvergenz,Nguetseng}, and extended to thin homogeneous structures in \cite{MarusicMarusicPalokaTwoScaleConvergenceThinDomains} respectively thin porous layers in \cite{GahnEffectiveTransmissionContinuous,NeussJaeger_EffectiveTransmission}. We prove two-scale compactness results for $W^{1,p}$-functions in thin perforated layers based on \textit{a priori} estimates including bounds for the symmetric gradient. The thin structure of the layer induces a different behavior in tangential and vertical direction, which leads in the two-scale limit to a Kirchhoff-Love displacement.  Further, for $p=2$ we identify the limit of the symmetric gradient, which involves an additional corrector function corresponding to the formal asymptotic expansion. Our proof is based on a periodic Helmholtz-decomposition for symmetric matrix valued functions.
Similar results for $p=2$ can be found in \cite{griso2020homogenization} (with slightly different lateral boundary conditions as mentioned above), where the decomposition for the displacement from \cite{griso2008decompositions} and the unfolding method \cite{CioranescuGrisoDamlamian2018} is used. Further, we have to mention the pioneering work \cite{caillerie1984thin}, which is to our knowledge the first result combining homogenization and dimension reduction for linear elastic problems in a plate. In \cite{Orlik2017} a limit model is derived using the unfolding method for a thin layer made up of thin vertical beams.
\\
As a proof of concept we apply our results to a semi-linear elastic wave equation in a clamped perforated thin layer with an inhomogeneous Neumann  boundary condition on the surface of the elastic substructure. Using standard energy estimates and the Korn-inequality for porous thin layers we derive \textit{a priori} estimates, which allow to apply the two-scale compactness results and pass to the limit $\epsilon \to 0$. For the treatment of the nonlinear term we use our extension operator together with the Aubin-Lions lemma to obtain strong compactness results. In the macroscopic model obtained for $\epsilon \to 0$, the displacement is described by a time-dependent plate equation with homogenized coefficients and including the second time-derivative of the vertical displacement. In the forthcoming paper \cite{GahnJaegerNeussRaduStokesPlate} a more complex problem for fluid-structure interaction for flow through a thin porous elastic layer is treated with the methods developed in this work. 
\\
The paper is organized as follows: In Section \ref{SectionMainResults} we introduce the geometrical setting and formulate the main results of the paper. In Section \ref{SectionExtensionOperator} we construct the extension operator. The Korn-inequality for perforated domains is shown in Section \ref{SectionKornInequality}, and the two-scale compactness results are given in Section \ref{SectionTwoScaleCompactness}. The homogonization of the semi-linear wave equation is done in Section \ref{SectionApplication}, followed by a conclusion in Section \ref{SectionConclusion}. In the Appendix \ref{SectionAppendix} we prove an auxiliary density result for periodic functions with weak divergence.

\section{Statement of the main results}
\label{SectionMainResults}
Let $p \in (1,\infty)$ and $p'$ denotes the dual exponent of $p$.
Let $\Sigma = (a,b)\subset \R^{n-1}$ with $a,b \in \Z^{n-1}$ and $a_i \kl b_i$ for $i=1,\ldots,n-1$. Further we assume that $\epsilon^{-1} \in \N$. We define the reference cell
\begin{align*}
Z := Y\times (-1,1) := (0,1)^{n-1} \times (-1,1),
\end{align*}
 and for $y \in Z$ we use the notation $\y:= (y_1,\ldots,y_{n-1})^T$. We consider a solid reference element $Z^s \subsetneq Z$, see Figure \ref{fig:Cell_Layer}, such that $Z^s $ is open and  connected with Lipschitz-boundary and  for $i=1,\ldots,n-1$
  \begin{align*}
 \mathrm{int} \left(\partial Z^s \cap \{y_i = 0\}\right)+e_i =  \mathrm{int} \left(\partial Z^s \cap \{y_i = 1\}\right).
 \end{align*}
\begin{figure}[h]
\begin{minipage}[t]{0.45\textwidth}
\centering
\includegraphics[scale=0.3]{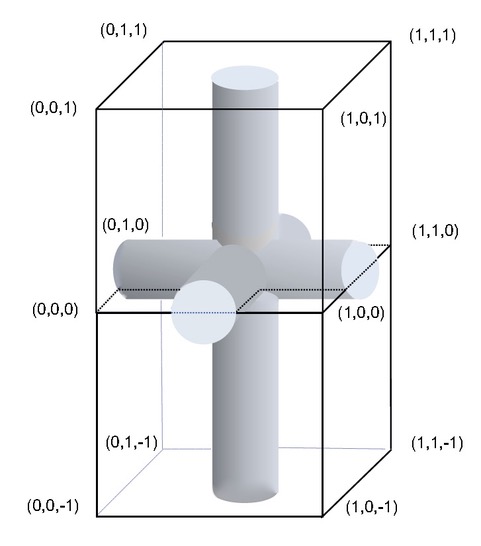}

\end{minipage}
\begin{minipage}[t]{0.49\textwidth}
\includegraphics[scale=0.30]{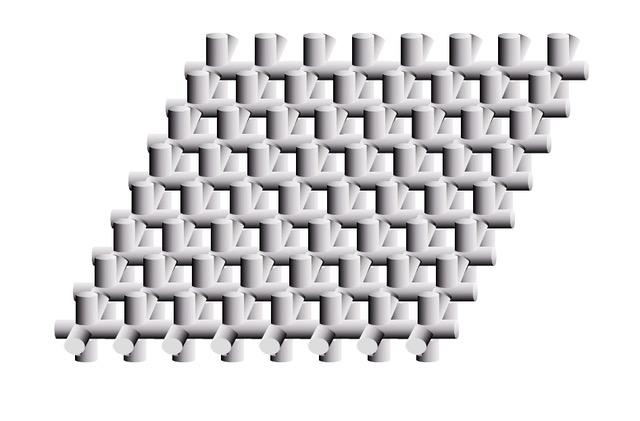}
\end{minipage}
\caption{Left: A solid reference cell $Z^s$ in grey. Right: The perforated layer $\oems$. }
 \label{fig:Cell_Layer}

\end{figure}

Further, we define $\Gamma: = \partial Z^s \setminus (\partial Y \times (-1,1))$. The complete thin layer $\oem$ is defined by
\begin{align*}
\oem := \Sigma \times (-\epsilon,\epsilon),
\end{align*}
and for $x \in \oem$ we write $\x := (x_1,\ldots,n-1)^T$.
We split the boundary of $\oem $ in its upper/lower part 
\begin{align*}
\sepm := \Sigma \times \{\pm 1 \},
\end{align*}
and its lateral part
\begin{align*}
\partial_D \oem:= \partial \Sigma \times (-\epsilon,\epsilon).
\end{align*}
Let $K_{\epsilon}:= \{k \in \Z^{n-1} \times \{0\} \, : \, \epsilon(Z + k) \subset \oem\}$. Clearly, we have $\oem = \mathrm{int}\left(\bigcup_{k\in K_{\epsilon}} \epsilon(\overline{Z} + k)\right)$. Now we define the perforated thin layer, see Figure \ref{fig:Cell_Layer}, by
\begin{align*}
\oems := \mathrm{int} \left(\bigcup_{k \in K_{\epsilon}} \epsilon\left(\overline{Z^s} + k\right)\right),
\end{align*}
and its lateral surface via
\begin{align*}
\partial_D \oems := \partial_D \oem \cap \partial \oems, 
\end{align*}
and its inner surface by
\begin{align*}
\geps := \partial \oems \setminus \overline{\partial_D \oems} = \mathrm{int}\left( \bigcup_{k\in K_{\epsilon}} \epsilon \left(\overline{\Gamma} + k\right)\right).
\end{align*}
By construction we have $\oems$ is connected and we assume that it is a Lipschitz-domain.
\\
For a weakly differentiable vector field $u$ we denote its symmetric gradient by
\begin{align*}
D(u):= \frac12 \left(\nabla u + \nabla u^T \right).
\end{align*}
Further, for a weakly differentiable function $u:\Sigma \rightarrow \R$, we use the notation $\nabla_{\x} u$ for both, a vector in $\R^2$, and also as the trivial embedding in $\R^3$ via $\nabla_{\x} u = (\partial_1 u , \partial_2 u, 0)^T$.
\\
Now, let us  summarize the main results of the paper. First of all, we have the following extension theorem which proof can be found in Section \ref{SectionExtensionOperator}.

\begin{theorem}\label{TheoremExtensionOperator}
There exists an extension operator $E_{\epsilon} : W^{1,p}(\oems)^n \rightarrow W^{1,p}(\oem)^n$, such that for all $\veps \in W^{1,p}(\oems)^n$ it holds that ($i=1,\ldots,n$)
\begin{align*}
\Vert (E_{\epsilon}\veps)^i \Vert_{L^p(\oem)} &\le C \left( \Vert \veps^i \Vert_{L^p(\oems)} + \epsilon \Vert \nabla \veps \Vert_{L^p(\oems)} \right),
\\
\Vert \nabla E_{\epsilon}\veps \Vert_{L^p(\oem)} &\le  C \Vert \nabla \veps \Vert_{L^p(\oems)},
\\
\Vert D(E_{\epsilon}\veps) \Vert_{L^p(\oem)} &\le C \Vert D(\veps)\Vert_{L^p(\oems)},
\end{align*}
for a constant $C\gr 0$ independent of $\epsilon$. 
\end{theorem}
This extension result is a key ingredient for obtaining the following 
 Korn-inequality for perforated thin layers with zero boundary conditions on $\partial_D \oems$ (For arbitrary boundary conditions see Proposition \ref{KornInequalityGeneralBoundaryRigidDisplacement}). The proof is given in Section \ref{SectionKornInequality} and is a direct consequence of the more general  result in Theorem \ref{KornInequalityScaledMembraneTheorem}.
\begin{theorem}\label{KornInequalityPerforatedLayer}
For all $\ueps \in W^{1,p}(\oems)^3$ with $\ueps = 0 $ on $ \partial_D \oems$ it holds that 
\begin{align*}
\sum_{i=1}^2 \foe \Vert \ueps^i &\Vert_{L^p(\oems)} + \sum_{i,j=1}^2 \foe\Vert \partial_i \ueps^j \Vert_{L^p(\oems)} 
\\
&+ \Vert \ueps^3\Vert_{L^p(\oems)} +  \Vert \nabla \ueps \Vert_{L^p(\oems)} \le \frac{C}{\epsilon} \Vert D(\ueps)\Vert_{L^p(\oems)}.
\end{align*}
Especially, this result holds in the trivial case $Z^s = Z$, where we have $\oems = \oem$.
\end{theorem}
The extension operator and the Korn-inequality for thin perforated layers enable us to prove the following compactness result with respect to the two-scale convergence (for the definition see Section \ref{SubsectionTSConvergence}), which gives a Kirchhoff-Love displacement in the limit:
\begin{theorem}\label{MainTheoremTwoScaleConvergence}
Let $\ueps\in W^{1,p}(\oems)^3$ be a sequence with
\begin{align*}
\Vert \ueps^3 \Vert_{L^p(\oem)} + \Vert \nabla \ueps \Vert_{L^p(\oem)} + \foe\Vert D(\ueps)\Vert_{L^p(\oem)} + \sum_{\alpha =1}^2 \foe\Vert \ueps^{\alpha} \Vert_{L^p(\oem)} \le C\epsilon^{\frac{1}{p}}.
\end{align*}
Then  there exist $u_0^3 \in W^{2,p}(\Sigma)$ and $\tilde{u}_1 \in W^{1,p}(\Sigma)^3$ with $\tilde{u}_1^3 = 0$, such that up to a subsequence (for $\alpha = 1,2$)
\begin{align*}
\chi_{\oems}\ueps^3 &\rats \chi_{Z^s}u_0^3,
\\
\chi_{\oems}\frac{\ueps^{\alpha}}{\epsilon} &\rats \chi_{Z^s}\big(\tilde{u}_1^{\alpha} - y_3 \partial_{\alpha} u_0^3\big).
\end{align*}
If additionally $\ueps = 0$ on $\partial_D \oems$, it holds that $u_0^3 \in W^{2,p}_0(\Sigma)$ and $\tilde{u}_1 \in W_0^{1,p}(\Sigma)^3$.
\\
Further, for $p = 2$ there exists $u_2 \in L^2(\Sigma , H^1_{\#}(Z)/\R)^3$, such that up to a subsequence 
\begin{align*}
\frac{1}{\epsilon} \chi_{\oems} D(\ueps) &\rats  \chi_{Z^s} \left(D_{\x}(\tilde{u}_1) - y_3 \nabla_{\x}^2 u_0^3 + D_y(u_2) \right).
\end{align*}
\end{theorem}
This result follows immediately from the Korn-inequality in Theorem \ref{KornInequalityPerforatedLayer} and the results in Section \ref{SectionTwoScaleCompactness}.

\begin{remark}
For the sake of simplicity we consider in Theorem \ref{KornInequalityPerforatedLayer} and \ref{MainTheoremTwoScaleConvergence} only the most important case $n=3$. However, the results can be extended in an obvious way to arbitrary space dimensions.
\end{remark}

\section{Extension operator}
\label{SectionExtensionOperator}

The aim of this section is to construct the extension operator $E_{\epsilon}$ from Theorem \ref{TheoremExtensionOperator}. While extension theorems  for perforated domains with estimates for the $L^p$-norm and the $L^p$-norm of the gradient are well-known, see \cite{Acerbi1992,CioranescuSJPaulin}, here
 the crucial point is the estimate for the symmetric gradient. The main ingredient of the proof is to show a local extension result of the same type for suitable reference elements. This idea can be found in \cite{Acerbi1992}. For the construction of this local operator we make use of special Korn-inequalities, which we summarize in the following.
%
%

\subsection{General Korn-inequalities in fixed domains}
Let $\Omega \subset \R^n$ be bounded, open, and connected.
We denote the space of rigid-displacements by $N(\Omega)$, which is defined via
\begin{align*}
N(\Omega):= \left\{ b + Ax \, : \, b \in \R^n, \, A \in \R^{n\times n} \mbox{ with } A + A^T = 0 \right\}.
\end{align*}
Further, we denote the antisymmetric gradient for $u \in W^{1,p}(\Omega)^n$ by
\begin{align*}
R(u) := \frac12 \left( \nabla u - \nabla u^T\right).
\end{align*}
Then we define the mean value of the antisymmetric gradient by
\begin{align}\label{MeanValueAntiSymGradient}
M(u):= \frac{1}{\vert \Omega \vert } \int_{\Omega} R(u) dx \in \R^{n\times n}.
\end{align}
\\

In the paper we frequently use the following Korn-inequalities:

\begin{lemma}\label{LemmaKornInequalities}
Let $\Omega \subset \R^n$ open, bounded, and connected with Lipschitz boundary.
\begin{enumerate}
[label = (\roman*)]
\item\label{LemmaKornInequalitiesMu} For all $u \in W^{1,p}(\Omega)^n$ it holds that
\begin{align*}
\Vert \nabla u - M(u) \Vert_{L^p(\Omega)} \le C \Vert D(u)\Vert_{L^p(\Omega)}.
\end{align*}
\item\label{LemmaKornInequalitiesWithL2Norm} For all $u \in W^{1,p}(\Omega)^n$ it holds that 
\begin{align*}
\Vert u \Vert_{W^{1,p}(\Omega)} \le C \left(\Vert u \Vert_{L^p(\Omega)} + \Vert D(u)\Vert_{L^p(\Omega)} \right).
\end{align*}
\item\label{LemmaKornInequalitiesRigidDisplacement} For every $u \in W^{1,p}(\Omega)^n$ there exists a rigid-displacement $r \in N(\Omega)$ (depending on $u$), such that
\begin{align*}
\Vert u - r \Vert_{W^{1,p}(\Omega)} \le C \Vert D(u)\Vert_{L^p(\Omega)}.
\end{align*}
\item\label{LemmaKornInequalitiesSubsetIntersectionRDZero} Let $V \subset W^{1,p}(\Omega)^n$ be a closed subspace with $V \cap N(\Omega) = \{0\}$.  Then every $u \in V$ fulfills 
\begin{align*}
\Vert u \Vert_{W^{1,p}(\Omega)} \le C \Vert D(u)\Vert_{L^p(\Omega)}.
\end{align*}
\item\label{LemmaKornInequalitiesZeroSubsetBoundary}
Let $\Gamma_0 \subset \partial \Omega$ with $\vert \Gamma_0 \vert \neq 0$ and such that 
\begin{align*}
\left\{v \in N(\Omega) \, : \, v = 0 \mbox{ on } \Gamma_0 \right\} = \{0\}.
\end{align*}
Then for all $u \in W^{1,p}(\Omega)^n$ with $u=0 $ on $\Gamma_0$ it holds that
\begin{align*}
\Vert u \Vert_{W^{1,p}(\Omega)} \le C \Vert D(u)\Vert_{L^p(\Omega)}.
\end{align*}
The assumptions are fulfilled for $\Omega = Z^s $ and $\Gamma_0 = \Gamma $ or $n=3$. 
\end{enumerate}
\end{lemma}
\begin{proof}
There exists an huge literature on different type of Korn-inequalities and the results stated in the lemma are quite common. However, for the sake of completeness  we give a sketch of the proof.
\\
Inequality \ref{LemmaKornInequalitiesMu} follows from $\nabla u = R(u) + D(u)$ and the inequality (which is closely related to arguments by Lions, see \cite[Section 3, Proof of Theorem 3.1]{duvant2012inequalities}) 
\begin{align}\label{InequalityErrorAntisymmetricGradient}
\Vert R(u) - M(u) \Vert_{L^p(\Omega)} \le C \Vert D(u)\Vert_{L^p(\Omega)} \qquad \mbox{for all } v \in W^{1,p}(\Omega)^n.
\end{align}
This is a consequence of the Ne\v{c}as-inequality, see \cite[Chapter 3, Lemma 7.1]{Necas}, 
\begin{align}\label{NecasInequality}
\Vert v \Vert_{L^p(\Omega)} \le C \Vert \nabla v \Vert_{H^{-1}(\Omega)}
\end{align}
for all $v \in L^p(\Omega)$ with mean value zero, applied to the antisymmetric gradient $R(u)$. In fact, for $i,j,k=1,\ldots,n$ it holds that
\begin{align*}
\partial_{x_j} R(u)_{ik} = \partial_{x_k} D(u)_{ji} - \partial_{x_i} D(u)_{jk} \qquad \mbox{in } H^{-1}(\Omega),
\end{align*}
and inequality $\eqref{NecasInequality}$ implies
\begin{align*}
\Vert R(u)_{ik}- M(u)_{ik} \Vert_{L^p(\Omega)} \le C \Vert \nabla R(u)_{ik}\Vert_{H^{-1}(\Omega)} \le C\Vert D(u)\Vert_{L^p(\Omega)}.
\end{align*}
\\
Statement \ref{LemmaKornInequalitiesWithL2Norm} can be shown in a similar way, see \cite[Section 3; Proof of Theorem 3.1]{duvant2012inequalities} for more details in the case $p=2$, which can be generalized in an obvious way.
\\
Let us prove \ref{LemmaKornInequalitiesRigidDisplacement}: Since $N(\Omega)$ has finite dimension there exists a closed subspace $\mathcal{M} \subset L^p(\Omega)^n$  such that
\begin{align*}
\mathcal{M} \oplus N(\Omega) = L^p(\Omega)^n.
\end{align*}
Further, there exists a linear and continuous projection operator $P:L^p(\Omega)^n \rightarrow N(\Omega)$. Now the result follows from a standard  contradiction argument and \ref{LemmaKornInequalitiesMu} (see also \cite[Section 3; Proof of Theorem 3.4]{duvant2012inequalities} for a similar argument).
%
\\
%
The result in \ref{LemmaKornInequalitiesSubsetIntersectionRDZero} can be found in \cite[Chapter I, Theorem 2.5]{Oleinik1992} for $p=2$ and the proof is also valid for arbitrary $p\in (1,\infty)$.
\ref{LemmaKornInequalitiesZeroSubsetBoundary} is a special case of \ref{LemmaKornInequalitiesSubsetIntersectionRDZero}.
\end{proof}

\subsection{The extension operator for perforated thin layers}

Now we construct the extension operator $E_{\epsilon}$ and prove the main extension Theorem \ref{TheoremExtensionOperator}.
We split the proof in several steps and start with a  local extension result, which is based on the Korn-inequality in Lemma \ref{LemmaKornInequalities}\ref{LemmaKornInequalitiesMu} including the mean $M(u)$ of the antisymmetric gradient. First of all, we introduce the following notations (see also Figure \ref{fig:Reference_Element_Extension}):
\begin{align*}
\E &:=  \{-1,0,1\}^{n-1} \times \{0\},
\\
Z_{\alpha} &:= Z + \alpha \qquad \mbox{for } \alpha \in \Z^{n-1}\times \{0\},
\\
Z^s_{\alpha} &:= Z^s + \alpha \qquad \mbox{for } \alpha \in \Z^{n-1}\times \{0\},
\\
K_{\epsilon} &:= \{\alpha \in \Z^{n-1} \times \{0\} \, : \,  \epsilon Z_{\alpha} \subset \oem \},
\\
\hZ_{\alpha} &:= \mathrm{int} \bigcup_{ e \in \E} \big( \overline{Z_{\alpha}} + e\big),
\\
\hZ^s_{\alpha} &:= \mathrm{int} \bigcup_{ e \in \E} \big( \overline{Z^s_{\alpha}} + e\big).
\end{align*}

\begin{figure}
\centering
\includegraphics[scale=0.4]{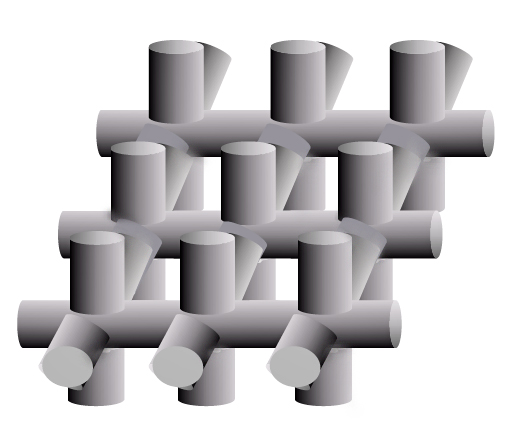}
\caption{A reference element $\hZ^s$ consisting of neighboring cells $Z^s$ for $n=3$}
 \label{fig:Reference_Element_Extension}
\end{figure}
We emphasize that $\hZ_{\alpha}$ (resp. $\hZ^s_{\alpha}$) is the cell $Z_{\alpha}$ with all neighboring cells. Further, we define the shift/scale function for $h \gr 0$ and $\alpha \in \Z^{n-1}\times \{0\}$
\begin{align*}
\pi^{\alpha}_h (x):=  \alpha + h x,
\end{align*}
and we shortly write
\begin{align*}
\pi^{\alpha}:= \pi^{\alpha}_1, \qquad \pi_h:= \pi^0_h.
\end{align*}

In the following lemma we construct a local extension operator on the reference element $\hZ^s$. This local operator is the main ingredient for the global extension operator on the perforated thin layer $\oems$.
\begin{lemma}\label{LocalExtensionOperator}
There exists an extension operator $\e : W^{1,p}(\hZ^s)^n \rightarrow W^{1,p}(\hZ)^n$, such that for all $v\in W^{1,p}(\hZ^s)^3$ it holds that ($i=1,\ldots,n$)
\begin{align*}
\Vert (\e v)^i \Vert_{L^p(\hZ)} &\le C\left(\Vert v^i \Vert_{L^p(\hZ^s)} + \Vert \nabla v \Vert_{L^p(\hZ^s)}\right),
\\
\Vert\nabla \e v \Vert_{L^p(\hZ)} &\le C \Vert \nabla v \Vert_{L^p(\hZ^s)},
\\
\Vert D(\e v)\Vert_{L^p(\hZ)} &\le C \Vert D(v) \Vert_{L^p(\hZ^s)}.
\end{align*}
\end{lemma}
\begin{proof}
Due to \cite{Acerbi1992}, there exists an extension operator $\tau : W^{1,p}(\hZ^s)^n \rightarrow W^{1,p}(\hZ)^n$ such that for all $v \in W^{1,p}(\hZ^s)^n$ it holds that ($i=1,\ldots,n$) 
\begin{align*}
\Vert (\tau v)^i \Vert_{L^p(\hZ)} &\le C \Vert v^i\Vert_{L^p(\hZ^s)},
\\
\Vert \nabla (\tau v)^i \Vert_{L^p(\hZ)} &\le C \Vert \nabla v^i \Vert_{L^p(\hZ^s)}.
\end{align*}
We define $\e : W^{1,p}(\hZ^s)^n \rightarrow W^{1,p}(\hZ)^n$ for  $v \in W^{1,p}(\hZ^s)^n$ by
\begin{align*}
\e v := \tau \big( v - M(v)y) + M(v)y,
\end{align*}
with the mean value of the antisymmetric gradient $M(v) \in \R^{n\times n}$ defined in $\eqref{MeanValueAntiSymGradient}$.
For $i=1,\ldots,n$ we have
\begin{align*}
\Vert (\e v)^i \Vert^p_{L^p(\hZ)} = \int_{\hZ} \vert (\e v)^i \vert^p dy
&\le C \int_{\hZ} \left\vert \tau(v - M(v)y)^i \right\vert^p + \vert M(v)y\vert^p dy 
\\
&\le C \int_{\hZ^s} \left\vert v_i - \big(M(v)y\big)^i \right\vert^p dy + C \vert M(v) \vert^p
\\
&\le C\left(\Vert v^i \Vert_{L^p(\hZ^s)}^p + \Vert \nabla v \Vert_{L^p(\hZ^s)}^p\right).
\end{align*}
For the gradient $\nabla \e v$ we can argue in a similar way.
It remains to estimate the symmetric gradient. Since $D(M(v)y) = \frac12 \big(M(v) + M(v)^T) = 0$, we obtain
\begin{align*}
\int_{\hZ} \vert D(\e v)\vert^p dy 
&\le C \int_{\hZ} \vert \nabla (\tau (v - M(v)y)) \vert^p dy 
\\
&\le C \int_{\hZ^s} \vert \nabla v - M(v) \vert^p dy 
\le C \int_{\hZ^s} \vert D(v) \vert^p dy,
\end{align*}
where at the end we used  the Korn-inequality  in Lemma \ref{LemmaKornInequalities}\ref{LemmaKornInequalitiesMu}.
\end{proof}

\begin{remark}\label{RemarkLocalExtensionOperator}
If we define the set 
\begin{align*}
\mathcal{Z} := \left\{ W = \mathrm{int} \bigcup_{\alpha \in I } \overline{Z_{\alpha}^s} \, : \, I \subset \E \cup \{0\}, \, W \mbox{ connected} \right\}, 
\end{align*}
the result of Lemma \ref{LocalExtensionOperator} is still valid if we replace $\hZ^s$ with an arbitrary $W\in \mathcal{Z}$  (see also \cite[Lemma 2.6]{Acerbi1992}). Especially, since $\mathcal{Z}$ is finite, we can choose the constant $C\gr 0$ in the inequalities the same for all $W$. We use the same notation $\e$ for all $W$ and neglect its dependence on the set $W$. 
\end{remark}

Next, we define an extension operator on the space $W^{1,p}(\epsilon^{-1} \oems)^n$ which preserves the norm of the symmetric gradient.
\begin{lemma}\label{ExtensionOperatorEpsilonScaledMembrane}
There exists an extension operator $E: W^{1,p}(\epsilon^{-1}\oems)^n \rightarrow W^{1,p}(\epsilon^{-1} \oem)^n $ (we neglect here the dependence on $\epsilon$) such that for all $\veps \in W^{1,p}(\epsilon^{-1} \oems)^n $ it holds that ($i=1,\ldots,n$)
\begin{align*}
\Vert (E \veps)^i \Vert_{L^p(\epsilon^{-1} \oem)} &\le C \left(\Vert \veps^i \Vert_{L^p(\epsilon^{-1}\oems)} + \Vert \nabla \veps \Vert_{L^p(\epsilon^{-1} \oems)}\right),
\\
\Vert \nabla E\veps \Vert_{L^p(\epsilon^{-1} \oem)} &\le C \Vert \nabla \veps \Vert_{L^p(\epsilon^{-1}\oems)},
\\
\Vert D(E\veps ) \Vert_{L^p(\epsilon^{-1} \oem)} &\le C \Vert D(\veps) \Vert_{L^p(\epsilon^{-1}\oems)},
\end{align*}
for a constant $C \gr 0 $ independent of $\epsilon$. 
\end{lemma}
\begin{proof}
As in \cite{Acerbi1992} we choose some kind of "periodic" partition of unity in the following way: Let $(\phi^{\alpha})_{\alpha \in \Z^{n-1}\times \{0\} }$ be a partition of unity of $(\hZ_{\alpha})_{\alpha \in \Z^{n-1}\times \{ 0 \} }$ , such that for all $\alpha , \beta \in \Z^{n-1} \times \{0\}$ it holds that
\begin{align*}
\phi^{\beta} = \phi^{\alpha} \circ \pi^{\alpha - \beta}.
\end{align*}
Especially, we have 
\begin{align*}
\sum_{e \in \E} \phi^{\alpha + e} = 1 \quad \mbox{ in } Z_{\alpha},
\\
\vert \phi^{\alpha} (x) \vert + \vert \nabla \phi^{\alpha} (x)\vert \le C
\end{align*}
for all $x \in \R^{n-1} \times (-1,1)$ and a constant $C\gr 0 $ independent of $\alpha$. 
To define a global extension operator, we also have to take account cells $\hZ_{\alpha}$ which are not completely included in $\epsilon^{-1} \oem$. For this we define
\begin{align*}
K_{\epsilon}^{\ast}:= \left\{\alpha \in \Z^{n-1} \times \{0\}\, : \, \hZ_{\alpha} \cap \epsilon^{-1} \oem \neq \emptyset \right\}.
\end{align*}
For $\veps \in W^{1,p}(\epsilon^{-1} \oems)^n$ and $\alpha \in K_{\epsilon}^{\ast}$ the function
\begin{align*}
\veps^{\alpha} := \left( \e (\veps\vert_{\hZ_{\alpha}\cap \epsilon^{-1} \oems} \circ \pi^{\alpha} )\right) \circ \pi^{-\alpha}.
\end{align*}
Now, we define the extension operator $E: W^{1,p}(\epsilon^{-1}\oems)^n \rightarrow W^{1,p}(\epsilon^{-1} \oem)^n $ via
\begin{align*}
E\veps := \sum_{\alpha \in K_{\epsilon}^{\ast}} \veps^{\alpha} \phi^{\alpha}.
\end{align*}
Obviously $E$ is a well-defined, linear, and bounded extension operator. From the product rule we obtain
\begin{align*}
D(\veps^{\alpha} \phi^{\alpha} ) = D(\veps^{\alpha})\phi^{\alpha} + \frac12\left(\nabla \phi^{\alpha} \otimes \veps^{\alpha} + \veps^{\alpha} \otimes \nabla \phi^{\alpha} \right).
\end{align*}
We have
\begin{align*}
\int_{\epsilon^{-1} \oem}& \vert D(E\veps)\vert^p dx = \sum_{\alpha \in K_{\epsilon}} \int_{Z_{\alpha}} \vert D(E\veps)\vert^p dx
\\
&= \sum_{\alpha \in K_{\epsilon} } \int_{Z_{\alpha}} \left\vert \sum_{e\in \E} \left[D(\veps^{\alpha + e})\phi^{\alpha+ e} + \frac12\left(\nabla \phi^{\alpha+ e} \otimes \veps^{\alpha+ e} + \veps^{\alpha+ e} \otimes \nabla \phi^{\alpha+ e} \right)   \right] \right\vert^p dx 
\\
&\le C \sum_{\alpha \in K_{\epsilon}} \bigg[\int_{Z_{\alpha}} \left\vert \sum_{e \in \E} D(\veps^{\alpha + e} ) \phi^{\alpha+ e}  \right\vert^p dx 
+ \int_{Z_{\alpha}} \left\vert \sum_{e \in \E} \nabla \phi^{\alpha + e } \otimes \veps^{\alpha + e } \right\vert^p dx
\\
&=: C \sum_{\alpha \in K_{\epsilon} } (A_{\epsilon,\alpha}^1 + A_{\epsilon,\alpha}^2).
\end{align*}
For the first term we obtain with Lemma \ref{LocalExtensionOperator} (for a constant $C\gr 0$ independent of $\epsilon$, $e$, and $\alpha$)
\begin{align*}
A_{\epsilon,\alpha}^1 \le C \sum_{e \in \E} \int_{\hZ} \left\vert D \left( \e(\veps \circ \pi^{\alpha + e} ) \right)\right\vert^p dx
&\le C \sum_{e \in \E} \int_{\hZ^s} \left\vert D\left(\veps \circ \pi^{\alpha +e} \right) \right\vert^p dx
\\
&= C \sum_{e \in \E} \int_{\hZ^s_{\alpha + e}} \vert D(\veps)\vert^p dx.
\end{align*}
Hence, we obtain
\begin{align*}
\sum_{\alpha \in K_{\epsilon} } A_{\epsilon,\alpha}^1 \le C \sum_{\alpha \in K_{\epsilon}^{\ast}} \int_{Z_{\alpha}^s} \vert D(\veps)\vert^p dx 
\le C \int_{\epsilon^{-1} \oems} \vert D(\veps)\vert^p dx.
\end{align*}
For $A_{\epsilon,\alpha}^2$ we get using $\sum_{e \in \E} \nabla \phi^{\alpha + e } = 0$
\begin{align*}
A_{\epsilon,\alpha}^2 &= \int_{Z_{\alpha}} \left\vert \sum_{e \in \E} \left[ \veps^{\alpha + e} - \veps^{\alpha} \right] \otimes \nabla \phi^{\alpha + e } \right\vert^p dx 
\le C \sum_{e \in \E} \int_{\hZ_{\alpha} \cap \hZ_{\alpha + e}} \vert \veps^{\alpha + e } - \veps^{\alpha} \vert^p dx.
\end{align*}
Since $\veps^{\alpha + e} - \veps^{\alpha} = 0$ on $\hZ^s_{\alpha}\cap \hZ^s_{\alpha + e}$ we can apply the Korn-inequality from Lemma  \ref{LemmaKornInequalities}\ref{LemmaKornInequalitiesZeroSubsetBoundary} (with a constant independent of $\alpha$, since we can transform the integral to the reference element) to get with similar arguments as above
\begin{align*}
A_{\epsilon,\alpha}^2 &\le C \sum_{e \in \E} \int_{\hZ_{\alpha} \cap \hZ_{\alpha + e}} \vert D(\veps^{\alpha + e}) - D(\veps^{\alpha})\vert^p dx
\\
&\le C \sum_{e \in \E} \left[ \int_{\hZ^s_{\alpha}} \vert D(\veps)\vert^p dx + \int_{\hZ^s_{\alpha + e}} \vert D(\veps)\vert^p dx \right].
\end{align*}
Summing over $\alpha $ we obtain
\begin{align*}
\sum_{\alpha \in K_{\epsilon} } A_{\epsilon,\alpha}^2 \le C \int_{\epsilon^{-1} \oems} \vert D(\veps)\vert^p dx.
\end{align*}
Altogether we obtain the desired result for $D(E\veps)$. With similar arguments we also obtain the estimate for $\nabla E\veps $ and $(E\veps)^i$, see also  \cite{Acerbi1992}.
\end{proof}

\begin{proof}[Proof of Theorem \ref{TheoremExtensionOperator}]
Let $\veps \in W^{1,p}(\oems)^n$. We define
\begin{align*}
E_{\epsilon} \veps := E(\veps \circ \pi_{\epsilon}) \circ \pi_{\epsilon^{-1}}.
\end{align*}
Then we have with Lemma \ref{ExtensionOperatorEpsilonScaledMembrane}
\begin{align*}
\int_{\oem} \vert D(E_{\epsilon} \veps ) \vert^p dx &= \int_{\oem} \vert D(E(\veps \circ \pi_{\epsilon})\circ \pi_{\epsilon^{-1}})\vert^p dx
\\
&= \epsilon^{n-p} \int_{\epsilon^{-1} \oem}  \vert D(E(\veps \circ \pi_{\epsilon}))\vert^p dx 
\\
&\le C \epsilon^{n-p} \int_{\epsilon^{-1}\oems} \vert D(\veps \circ \pi_{\epsilon}) \vert^p dx
\\
&= C \int_{\oems} \vert D(\veps)\vert^p dx.
\end{align*}
The inequality for $\nabla E_{\epsilon} \veps$ is obtained in a similar way. Further, for $i=1,\ldots,n$ we have
\begin{align*}
\int_{\oem} \vert (E_{\epsilon} \veps)^i\vert^p dx &=  \epsilon^n \int_{\epsilon^{-1} \oem} \vert (E(\veps \circ \pi_{\epsilon}))^i\vert^p dx 
\\
&= C \epsilon^n \int_{\epsilon^{-1} \oems} \vert \veps^i \circ \pi_{\epsilon} \vert^p + \epsilon^p \vert \nabla \veps \circ \pi_{\epsilon}\vert^p dx
\\
&\le C \left( \Vert \veps^i \Vert_{L^p(\oems)}^p + \epsilon^p \Vert \nabla \veps \Vert_{L^p(\oems)}^p \right).
\end{align*}
This finishes the proof.
\end{proof}

\section{Korn-inequality in the perforated thin layer}
\label{SectionKornInequality}

We derive the Korn-inequality for thin perforated domains in Theorem \ref{KornInequalityPerforatedLayer}. The crucial point is the explicit dependence on the parameter $\epsilon$. The idea is to transform the problem to the fixed rescaled domain $\Omega_1^M$ and apply the Korn-inequalities from Lemma \ref{LemmaKornInequalities}. For this we have to extend functions from the perforated layer $\oems$ to the whole domain $\oem$ by using Theorem \ref{TheoremExtensionOperator} which gives a control of the symmetric gradient of the extension. An additional problem arises due to the fact, that the rescaled osciallating lateral boundary $\partial_D \oems$ is depending on $\epsilon$ and therefore not fixed (an application of Lemma \ref{LemmaKornInequalities}\ref{LemmaKornInequalitiesZeroSubsetBoundary} uniformly with respect to $\epsilon $ is not possible). To overcome this problem we use a contradiction argument.
\\
%
For $v \in W^{1,p}(\Omega_1^M)^3$ we define the $\epsilon$-weighted symmetric gradient $\kappa_{\epsilon}(v)$ by ($i,j=1,2$)
\begin{align*}
\kappa_{\epsilon}(v)_{ij} = D(v)_{ij}, \quad 
\kappa_{\epsilon}(v)_{i3} = \foe D(v)_{i3}, \quad 
\kappa_{\epsilon}(v)_{33} = \frac{1}{\epsilon^2} D(v)_{33}.
\end{align*}
For an arbitrary domain $G_{\epsilon} \subset \R^3$  we define on $W^{1,p}(G_{\epsilon})^3$ the norm
\begin{align*}
\Vert \veps \Vert_{G_{\epsilon},\epsilon}^p &:= \sum_{i=1}^2 \frac{1}{\epsilon^{p+1}} \Vert \veps^i \Vert_{L^p(\oem)}^p + \sum_{i,j=1}^2 \frac{1}{\epsilon^{p+1}}\Vert \partial_i \veps^j \Vert_{L^p(\oem)}^p + \frac{1}{\epsilon} \Vert \veps^3 \Vert_{L^p(\oem)}^p 
\\
&+ \epsilon\Vert \partial_3 \veps^3 \Vert_{L^p(\oem)}^p+ \sum_{i=1}^2 \frac{1}{\epsilon} \left(\Vert \partial_3 \veps^i \Vert_{L^p(\oem)}^p + \Vert \partial_i \veps^3 \Vert_{L^p(\oem)}^p \right)
\end{align*}
for all $\veps \in W^{1,p}(G_{\epsilon})^3$. Now, for every $\ueps \in W^{1,p}(\oem)^3$ we define the function $\tueps: \Omega_1^M \rightarrow \R^3$ by ($i=1,2$)
\begin{align*}
\tueps^i(x):= \frac{1}{\epsilon} \ueps^i(\x,\epsilon x_3),\qquad \tueps^3(x):= \ueps^3(\x,\epsilon x_3).
\end{align*}
An elemental calculation gives
\begin{align}
\begin{aligned}
\label{IdentityNormsScaledMembrane}
\Vert \tueps \Vert_{W^{1,p}(\Omega_1^M)} &= \Vert \ueps\Vert_{\oem,\epsilon},
\\
\epsilon^{1 + \frac{1}{p}}\Vert \kappa_{\epsilon} (\tueps) \Vert_{L^p(\Omega_1^M)} &= \Vert D(\ueps)\Vert_{L^p(\oem)}. 
\end{aligned}
\end{align}

We have the following Korn-inequality in the thin perforated domain  (see also \cite[Theorem 5]{griso2020homogenization} for a similar result for the case $p=2$):

\begin{proposition}\label{KornInequalityGeneralBoundaryRigidDisplacement}
For all $\ueps \in W^{1,p}(\oems)^3$ there exists a rigid-displacement $\reps$, such that
\begin{align*}
\sum_{i=1}^2 &\foe \Vert \ueps^i - \reps^i\Vert_{L^p(\oems)} + \sum_{i,j=1}^2 \foe \Vert \partial_i \ueps^j - \partial_i \reps^j \Vert_{L^p(\oems)} 
\\
& + \Vert \ueps^3 -\reps^3\Vert_{L^p(\oems)} + \Vert \nabla \ueps - \nabla \reps \Vert_{L^p(\oems)} \le \frac{C}{\epsilon} \Vert D(\ueps)\Vert_{L^p(\oems)}.
\end{align*}
\end{proposition}
\begin{proof}
We extend $\ueps \in W^{1,p}(\oems)^3$ with the extension operator $E_{\epsilon}$ from Theorem \ref{TheoremExtensionOperator} to the whole layer $\oem$ and denote this extension again by $\ueps$. From the Korn-inequality in Lemma \ref{LemmaKornInequalities}\ref{LemmaKornInequalitiesRigidDisplacement} we get  the existence of a rigid-displacement $\treps$, such that
\begin{align*}
\Vert \tueps - \treps\Vert_{W^{1,p}(\Omega_1^M)} \le C \Vert D(\tueps)\Vert_{L^p(\Omega_1^M)}.
\end{align*}
We define the rigid-displacement $\reps$ on $\oem$ by ($\alpha = 1,2$)
\begin{align*}
\reps^{\alpha}(x) := \epsilon \treps\left(\x,\frac{x_3}{\epsilon}\right), \qquad \reps^3(x):= \treps\left(\x,\frac{x_3}{\epsilon}\right).
\end{align*}
Using the identity $\eqref{IdentityNormsScaledMembrane}$ and Theorem \ref{TheoremExtensionOperator}, we obtain
\begin{align*}
\Vert \ueps - \reps  \Vert_{\oems,\epsilon} &\le \Vert \ueps - \reps \Vert_{\oem,\epsilon} 
= \Vert \tueps - \treps \Vert_{W^{1,p}(\Omega_1^M)} 
\le C \Vert D(\tueps)\Vert_{L^p(\Omega_1^M)}
\\
&\le C\Vert \kappa_{\epsilon}\Vert_{L^p(\Omega_1^M)}
= \frac{C}{\epsilon^{1 + \frac{1}{p}}} \Vert D(\ueps)\Vert_{L^p(\oem)}
\le \frac{C}{\epsilon^{1 + \frac{1}{p}}} \Vert D(\ueps)\Vert_{L^p(\oems)}.
\end{align*}
This implies the desired result. We emphasize that the norm $\Vert \cdot \Vert_{\oems,\epsilon}$ includes another $\epsilon$-weight for the component $\partial_3 \ueps^3$ than stated in the claim of the Proposition. However, since this component is included in $D(\ueps)$, the claim follows.
\end{proof}
A crucial question is under which conditions we have $\reps = 0$ in Proposition \ref{KornInequalityGeneralBoundaryRigidDisplacement}. Of course, the Korn-inequality in Lemma \ref{LemmaKornInequalities}\ref{LemmaKornInequalitiesZeroSubsetBoundary} implies for all $\ueps \in W^{1,p}(\oems)^3$ with $\ueps = 0$ on $\partial_D \oems$
\begin{align*}
\Vert \tueps \Vert_{W^{1,p}(\Omega_1^M)} \le C_{\epsilon} \Vert D(\tueps)\Vert_{L^p(\Omega_1^M)}.
\end{align*}
However, the constant $C_{\epsilon}$ may depend on $\epsilon$. If $\tueps = 0$ on the lateral boundary $\partial \Sigma \times (-1,1)$ the constant $C_{\epsilon}$ may be chosen independently of $\epsilon$, see also \cite[Proof of Theorem 1.4-1, page 36]{ciarlet1997mathematical} for $p=2$. In general, this is not fulfilled by the extension operator from Theorem \ref{TheoremExtensionOperator}. The following lemma  gives a bound for the $L^p$-norm of the trace on the whole lateral boundary $\partial_D \oem$ for functions vanishing on the perforated part $\partial_D \oems$.
\begin{lemma}\label{TraceInequalityOuterBoundaryLemma}
For every $\ueps \in W^{1,p}(\oem)^3$ with $\ueps = 0$ on $\partial_D \oems$ it holds that
\begin{align*}
\Vert \ueps \Vert_{L^p(\partial \Sigma \times (-\epsilon,\epsilon)) } \le C \epsilon^{\frac{1}{p}}\Vert D(\ueps)\Vert_{L^p(\oem)}.
\end{align*} 
\end{lemma}
\begin{proof}
We define 
\begin{align*}
K_{\epsilon}^b := \{\alpha \in K_{\epsilon}\, : \,  \epsilon \overline{Z_{\alpha}} \cap (\partial \Sigma \times (-\epsilon,\epsilon)) \neq \emptyset\}.
\end{align*}
Using the trace-inequality and the Korn-inequality in Lemma \ref{LemmaKornInequalities}\ref{LemmaKornInequalitiesZeroSubsetBoundary} on $Z$ (what is possible, since we have the zero boundary condition of $\ueps(\epsilon(y+k))$ on one side of cell $Z$), we obtain
\begin{align*}
\Vert \ueps \Vert_{L^p(\partial \Sigma \times (-\epsilon,\epsilon))}^p 
&\le C \epsilon^{n-1} \sum_{k \in K_{\epsilon}^b} \Vert \ueps (\epsilon(y + k))\Vert_{W^{1,p}(Z)}^p
\\
&\le  C \epsilon^{n-1} \sum_{k \in K_{\epsilon}^b} \Vert D(\ueps(\epsilon(y+k)))\Vert_{L^p(Z)}^p
\\
&\le C\epsilon \Vert D(\ueps)\Vert_{L^p(\oem)}^p.
\end{align*}
\end{proof}

Now we are able to prove a Korn-inequality for functions vanishing on the perforated lateral boundary $\partial_D \oems$. We use similar ideas as in \cite[Proof of Theorem 2.3]{lewicka2011uniform}.

\begin{theorem}\label{KornInequalityScaledMembraneTheorem}
For every $\ueps \in W^{1,p}(\oem)^3 $ with $\ueps = 0$ on $\partial_D \oems$ it holds that
\begin{align*}
\Vert \tueps \Vert_{W^{1,p}(\Omega_1^M)} \le C \Vert \kappa_{\epsilon}(\tueps) \Vert_{L^p(\Omega_1^M)}.
\end{align*}
\end{theorem}



\begin{proof}
We prove this result by a contradiction argument. Let us assume  the inequality is not valid. Then (due to the Korn-inequality from Lemma \ref{LemmaKornInequalities}\ref{LemmaKornInequalitiesZeroSubsetBoundary}), there exists a subsequence $\epsilon = \epsilon_l \rightarrow 0$ and $C_{\epsilon}\gr 0$ with $C_{\epsilon}\rightarrow \infty$, as well as $\ueps \in W^{1,p}(\oem)^3$ with $\ueps  = 0$ on $\partial_D \oems $, such that 
\begin{align}\label{AuxiliaryEstimateContradictionKorn}
1 = \Vert \tueps \Vert_{W^{1,p}(\Omega_1^M)} \geq C_{\epsilon} \Vert \kappa_{\epsilon}(\tueps)\Vert_{L^p(\Omega_1^M)}.
\end{align}
In the following the constant $C_{\epsilon}$ is a generic constant depending on $\epsilon$, such that we always have $C_{\epsilon}\rightarrow \infty$ for $\epsilon \to 0$. 
The boundedness of $\tueps$ in $W^{1,p}(\Omega_1^M)^3$ implies the existence of a subsequence and $u_0 \in W^{1,p}(\Omega_1^M)^3$, such that
\begin{align*}
\tueps &\rightharpoonup u_0 &\mbox{ weakly in }& W^{1,p}(\Omega_1^M)^3,
\\
\tueps &\rightarrow u_0 &\mbox{ strongly in }& L^p(\Omega_1^M)^3,
\\
\tueps\vert_{\partial \Omega_1^M} &\rightarrow u_0\vert_{\partial \Omega_1^M} &\mbox{ strongly in }& L^p(\partial \Omega_1^M)^3.
\end{align*}
Further, Lemma \ref{TraceInequalityOuterBoundaryLemma} implies ($\alpha = 1,2$)
\begin{align*}
\Vert \tueps^{\alpha}\Vert_{L^p(\partial \Sigma \times (-1,1))} &\le \frac{C}{\epsilon} \Vert D(\ueps)\Vert_{L^p(\oem)} = C \epsilon^{\frac{1}{p}}  \Vert \kappa_{\epsilon} (\tueps)\Vert_{L^p(\Omega_1^M)} \overset{\epsilon\to 0}{\longrightarrow} 0,
\\
\Vert  \tueps^3 \Vert_{L^p(\partial \Sigma \times (-1,1))} &\le C \Vert D(\ueps)\Vert_{L^p(\oems)} = C\epsilon^{1 + \frac{1}{p}} \Vert \kappa_{\epsilon} (\tueps)\Vert_{L^p(\Omega_1^M)} \overset{\epsilon\to 0}{\longrightarrow} 0.
\end{align*}
Hence, we obtain $u_0 = 0 $ on $\partial \Sigma \times (-1,1)$.

The Korn-inequality in Lemma \ref{LemmaKornInequalities}\ref{LemmaKornInequalitiesRigidDisplacement} implies the existence of a rigid displacement $\reps(x) = b_{\epsilon} + A_{\epsilon}x $ with $b_{\epsilon} \in \R^3$ and $A_{\epsilon} \in \R^{3 \times 3} $ skew-symmetric, such that (see again $\eqref{AuxiliaryEstimateContradictionKorn}$)
\begin{align}\label{ProofKornAuxiliaryKornInequality}
\Vert \tueps - \reps\Vert_{W^{1,p}(\Omega_1^M)} \le C \Vert D(\tueps)\Vert_{L^p(\Omega_1^M)}\le C \Vert \kappa_{\epsilon}(\tueps)\Vert_{L^p(\Omega_1^M)} \overset{\epsilon \to 0}{\longrightarrow}  \ 0.
\end{align}
From   $\eqref{AuxiliaryEstimateContradictionKorn}$ we also obtain 
\begin{align}
\label{ConvergenceNormRigidDisplacement}
\lim_{\epsilon \to 0} \Vert \reps \Vert_{W^{1,p}(\Omega_1^M)} = 1.
\end{align}
Especially, $\reps$ is bounded in $W^{1,p}(\Omega_1^M)$ and therefore $b_{\epsilon}$ is bounded in $\R^3$ and $A_{\epsilon} $ is bounded in $\R^{3\times 3}$. Hence, there exists $b \in \R^3$ and $A\in \R^{3\times 3}$ skew-symmetric, such that up to a subsequence $b_{\epsilon} \rightarrow b $ and $A_{\epsilon} \rightarrow A $.
Further, $\eqref{ConvergenceNormRigidDisplacement}$ implies
\begin{align}\label{ProofKornRigidDisplLimitNotZero}
\vert A \vert + \vert b\vert \neq 0.
\end{align}
Due to $\eqref{ProofKornAuxiliaryKornInequality}$, we get 
\begin{align*}
u_0(x) = b + Ax \qquad \mbox{ in } \Omega_1^M.
\end{align*}
Since $u_0$ vanishes on the lateral boundary, we obtain $b=0$ and $A=0$, which contradicts $\eqref{ProofKornRigidDisplLimitNotZero}$.

\end{proof}

As an immediate consequence we obtain with the extension operator from Theorem \ref{TheoremExtensionOperator} the Korn-inequality in Theorem \ref{KornInequalityPerforatedLayer}:

\begin{proof}[Proof of Theorem \ref{KornInequalityPerforatedLayer}]
Using the extension operator $E_{\epsilon}$ from Theorem \ref{TheoremExtensionOperator} and the identities in $\eqref{IdentityNormsScaledMembrane}$, we obtain
\begin{align*}
\Vert \ueps\Vert_{\oems,\epsilon} &\le \Vert E_{\epsilon}\ueps \Vert_{\oem,\epsilon} = \Vert \widetilde{E_{\epsilon} \ueps} \Vert_{W^{1,p}(\Omega_1^M)}
\\
&\le C \Vert \kappa_{\epsilon}(\widetilde{E_{\epsilon}\ueps })\Vert_{L^p(\Omega_1^M)} =C \epsilon^{-1- \frac{1}{p}}\Vert D(E_{\epsilon}\ueps)\Vert_{L^p(\oem)} 
\\
&\le C \epsilon^{-1- \frac{1}{p}}\Vert D(\ueps)\Vert_{L^p(\oems)}.
\end{align*}
This estimate implies the desired result.
\end{proof}

%

\section{Two-scale compactness results}
\label{SectionTwoScaleCompactness}

In this section we derive two-scale compactness results  in thin perforated layers based  on uniform bounds for the symmetric gradient. Such results are of particular importance in continuum mechanics, especially in linear elasticity. Our results take into account both the dimension reduction of the thin layer and the homogenization process due to the heterogeneous structure of the thin layer. The thin structure of the layer induces a different behavior in $\x$- and $x_3$ direction, which leads in the limit to a Kirchhoff-Love displacement. We will see that due to the extension operator from Theorem \ref{TheoremExtensionOperator} we can treat $W^{1,p}$-functions on the perforated layer as functions defined on the whole layer. Further, the Korn-inequality in Theorem \ref{KornInequalityPerforatedLayer} guarantees a control of an $\epsilon$-weighted $W^{1,p}$-norm on the perforated layer by the symmetric gradient.
\\
In the following function spaces with the index $\#$ denote spaces which are $Y$-periodic. Especially
we define the space of smooth and $Y$-periodic functions 
\begin{align*}
C_{\#}^{\infty}(\overline{Z}):= \left\{ v \in C^{\infty}(\R^2 \times [-1,1]) \, : \, v(\cdot + e_i ) = v \mbox{ for } i=1,2\right\},
\end{align*}
and  $W^{1,p}_{\#}(Z)$ is the closure of $C_{\#}^{\infty}(\overline{Z})$ with respect to the usual $W^{1,p}$-norm. The space $W^{1,p}_{\#}(Z^s)$ is defined by restriction of functions from $W^{1,p}_{\#}(Z)$.

\subsection{The two-scale convergence}
\label{SubsectionTSConvergence}
We briefly summarize the concept of two-scale convergence in thin layers \cite{GahnNeussRadu2017EffectiveTransmissionConditions,NeussJaeger_EffectiveTransmission}, which takes into account the simultaneous limit process for the   homogenization and the dimension reduction.  

\begin{definition}
We say the sequence $\veps \in L^p(\oem)$ converges (weakly) in the two-scale sense to a limit function $v_0 \in L^p(  \Sigma \times Z)$ if 
\begin{align*}
\lim_{\epsilon\to 0} \foe \int_{\oem} \veps(x) \phi \bxfxe dx = \int_{\Sigma} \int_Z v_0(\x,y) \psi(\x,y) dy d\x
\end{align*}
for all $\phi \in L^p(\Sigma,C_{\#}^0(\overline{Z}))$. We write 
\begin{align*}
\veps \rats v_0.
\end{align*}
We say the sequence converges strongly in the two-scale sense if additionally 
\begin{align*}
\lim_{\epsilon\to 0}  \epsilon^{-\frac{1}{p}} \Vert \veps \Vert_{L^p(\oem)} = \Vert v_0 \Vert_{L^p(\Sigma\times Z)}.
\end{align*}
\end{definition}
We have the following compactness results, see \cite{GahnNeussRadu2017EffectiveTransmissionConditions,NeussJaeger_EffectiveTransmission}.
\begin{lemma}\label{TwoScaleCompactnessStandard}\
\begin{enumerate}
[label = (\roman*)]
\item For every sequence $\veps$ in $L^p(\oem)$ with
\begin{align*}
 \Vert \veps \Vert_{L^p(\oem)} \le C \epsilon^{\frac{1}{p}},
\end{align*}
there exists a (weak) two-scale convergent subsequence.
\item Let $\veps \in W^{1,p}(\oem)$ be a sequence with 
\begin{align*}
  \Vert \veps \Vert_{W^{1,p}(\oem)} \le C\epsilon^{\frac{1}{p}}.
\end{align*} 
Then there exist $v_0 \in W^{1,p}(\Sigma)$ and $v_1 \in L^p(\Sigma,W^{1,p}_{\#}(Z)/\R)$ such that up to a subsequence
\begin{align*}
\veps &\rats v_0 ,
\\
\nabla \veps &\rats \nabla v_0 + \nabla_y v_1.
\end{align*}
\end{enumerate}

\end{lemma}

\begin{remark}
The two-scale convergence can be generalized to time-dependent functions. In this case the time-variable acts as a parameter and the results in Lemma \ref{TwoScaleCompactnessStandard} remain valid. For more details we refer to \cite{GahnNeussRadu2017EffectiveTransmissionConditions,NeussJaeger_EffectiveTransmission}.
\end{remark}

\subsection{Compactness results with bounded symmetric gradient}

We derive general two-scale compactness results which are based on uniform estimates of the symmetric gradient. The Korn-inequality from Theorem \ref{KornInequalityPerforatedLayer} guarantees uniform bounds for the functions and the gradient, and due to the extension Theorem \ref{TheoremExtensionOperator} functions in the perforated layer can be treated as a function defined in the whole layer. Similar results with proofs based on the unfolding method can be found in \cite{griso2020homogenization} for $p=2$.

\subsubsection{The Kirchhoff-Love displacement as a two-scale limit}

We start with showing a simple auxiliary lemma which together with the \textit{a priori } bound for the symmetric gradient and the specific structure of the two-scale limit of the gradient in Lemma \ref{TwoScaleCompactnessStandard} already gives a Kirchhoff-Love displacement in the limit.

\begin{lemma}\label{LemmaKirchhoffZerlegungSymGra}
Let $u_0 \in W^{1,p}(\Sigma)^3$ and $u_1 \in L^p(\Sigma,W^{1,p}_{\#}(Z^s)/\R)^3$, such that 
\begin{align*}
D_{\x} (u_0) + D_y(u_1) = 0.
\end{align*}
Then, there exists $\tilde{u}_1 \in L^p(\Sigma)^3$ such that
\begin{align*}
u_1(\x,y) &= \tilde{u}_1 - y_3 \begin{pmatrix}
\partial_1 u_0^3 \\ \partial_2 u_0^3 \\ 0
\end{pmatrix}.
\end{align*}
Further, we obtain 
\begin{align*}
\partial_1 u_0^1 = \partial_2 u_0^2 = \partial_1 u_0^2 + \partial_2 u_0^1 = 0.
\end{align*}
\end{lemma}
\begin{proof}
For every matrix $A \in \R^{3 \times 3}$ it holds that $D_y (Ay) = \frac12 (A + A^T). $
Hence, we obtain
\begin{align*}
0 = D_{\x}(u_0) + D_y(u_1) = D_y(D_{\x}(u_0)y + u_1).
\end{align*}
This implies that $D_{\x}(u_0)y + u_1$ is a rigid-displacement with respect to $y$, i.e. there exist $\tilde{u}_1 \in L^p(\Sigma)^3$ (we will check the regularity below) and a skew-symmetric matrix $R(\x)$ such that
\begin{align*}
u_1(\x,y) = \tilde{u}_1(\x) + R(\x)y - D_{\x}(u_0)y.
\end{align*}
Since $u_1$ is periodic with respect to $y_1$ and $y_2$ we obtain for the unit vectors $e_i$ for $i=1,2$ that
\begin{align}\label{IdentitySymAntisymMatrix}
R(\x) e_i = D_{\x}(u_0)e_i.
\end{align}
This implies with the symmetry of $D_{\x}(u_0)$ and the antisymmetry of $R(\x)$ for $i,j=1,2$
\begin{align*}
R_{ij} = D_{\x}(u_0)_{ij} = D_{\x}(u_0)_{ji} = R_{ji} = -R_{ij}
\end{align*}
and therefore $R_{ij} = 0 = D_{\x}(u_0)_{ij}$. Further $\eqref{IdentitySymAntisymMatrix} $ implies $R_{3l} = D_{\x}(u_0)_{3l}$ for $l=1,2$.  This immediately implies 
%
%
%
\begin{align*}
\partial_1 u_0^1 = \partial_2 u_0^2 = \partial_1 u_0^2 + \partial_2 u_0^1 = 0.
\end{align*}
and with a simple calculation
\begin{align*}
u_1(\x,y) = \tilde{u}_1(\x) - y_3 \nabla_{\x} u_0^3(\x).
\end{align*}
This especially implies the $L^p$-regularity of $\tilde{u}_1$.
%
\end{proof}

In the next proposition we show that a sequence $\ueps$ with $L^p$-norm of the symmetric gradient $D(\ueps)$ of order $\epsilon^{1 + \frac{1}{p}}$ can be approximated (in the two-scale sense) by a Kirchhoff-Love displacement.

\begin{proposition}\label{CompactnessTwoScaleThinLayer}
Let $\ueps \in W^{1,p}(\oem)^3$ be a sequence with
\begin{align}\label{AssumptionInequalityCompactnessThinLayer}
\Vert \ueps^3 \Vert_{L^p(\oem)} + \Vert \nabla \ueps \Vert_{L^p(\oem)} + \foe\Vert D(\ueps)\Vert_{L^p(\oem)} + \sum_{\alpha =1}^2 \foe\Vert \ueps^{\alpha} \Vert_{L^p(\oem)} \le C\epsilon^{\frac{1}{p}}.
\end{align}
Then there exists $u_0^3 \in W^{2,p}(\Sigma)$ and $\tilde{u}_1 \in W^{1,p}(\Sigma)^3$ with $\tilde{u}_1^3 = 0$, such that up to a subsequence (for $\alpha = 1,2$)
\begin{align*}
\ueps^3 &\rats u_0^3,
\\
\frac{\ueps^{\alpha}}{\epsilon} &\rats \tilde{u}_1^{\alpha} - y_3 \partial_{\alpha} u_0^3.
\end{align*}
\end{proposition}
\begin{proof}
From the two-scale compactness results in Lemma \ref{TwoScaleCompactnessStandard} we obtain the existence of $u_0^3 \in W^{1,p}(\Sigma)$, $u_1\in L^p(\Sigma,W^{1,p}_{\#}(Z)/\R)^3$, and $\eta^{\alpha} \in L^p(\Sigma\times Z)$, such that up to a subsequence
\begin{align*}
\ueps^3 &\rats u_0^3,
\\
\nabla \ueps &\rats \nabla_{\x} \left(
0 , 0 , u_0^3 
\right)^T + \nabla_y u_1,
\\
\frac{\ueps^{\alpha}}{\epsilon} &\rats \eta^{\alpha}.
\end{align*}
We emphasize that we have chosen $u_1$ in such a way that its mean value with respect to $y$ is zero.
From Lemma \ref{LemmaKirchhoffZerlegungSymGra} we obtain the existence of $\tilde{u}_1 \in L^p(\Sigma)^3$ 
\begin{align*}
u_1(\x,y) = \tilde{u}_1(\x) - y_3 \nabla_{\x} u_0^3(\x).
\end{align*}
Since $u_1$ has mean zero we get $\tilde{u}_1^3 = 0$.
Now, let $\phi \in C_0^{\infty}(\Sigma)$ and $\psi \in C_0^{\infty}(Z)$ with $\int_Z \psi dy = 0$. Hence, using the Bogosvkii-operator, see for example  \cite[Theorem III.3.3]{Galdi}, $\psi $ can be represented as the divergence of a smooth function with compact support in $Z$. With this we  obtain, for more details we refer to \cite[Theorem 2.3]{DouanlaTwoScaleConvergence},
\begin{align}
\begin{aligned}
\label{ConvergenceTestFunctionsMeanZero}
\lim_{\epsilon\to 0} \foe \int_{\oem} \frac{\ueps^{\alpha}(x)}{\epsilon} \phi(\x) \psi\left(\fxe\right) dx &= \int_{\Sigma}\int_Z \eta^{\alpha} (\x,y) \phi(\x) \psi(y) dy d\x,
\\
\lim_{\epsilon\to 0} \foe \int_{\oem} \frac{\ueps^{\alpha}(x)}{\epsilon} \phi(\x) \psi\left(\fxe\right) dx &= \int_{\Sigma} \int_Z  u_1^{\alpha}(\x,y) \phi(\x) \psi(y) dy d\x.
\end{aligned}
\end{align}
Hence, the functions $u_1^{\alpha}$ and $\eta^{\alpha}$ only differ by a "constant" depending on $\x$. Therefore, we can change $\tilde{u}_1^{\alpha}$ in such a way (we use again the same notation) that
\begin{align*}
\eta^{\alpha}(\x,y) = \tilde{u}_1^{\alpha} - y_3 \partial_{\alpha} u_0^3.
\end{align*} 
It remains to show the regularity for $u_0^3$ and $\tilde{u}_1$. We choose a decomposition from \cite{griso2008decompositions} but with a  different scaling with respect to $\epsilon$:
\begin{align*}
\U_{\epsilon}(\x):= \frac{1}{2\epsilon^2 } \int_{-\epsilon}^{\epsilon} \begin{pmatrix}
\ueps^1 \\ \ueps^2 
\end{pmatrix}(\x,x_3) dx_3.
\end{align*}
We have $\U_{\epsilon} \in W^{1,p}(\Sigma)^2$ with  ($\alpha , \beta = 1,2$)
\begin{align*}
\Vert \U_{\epsilon}^{\alpha} \Vert_{L^p(\Sigma)} \le C \epsilon^{-1 - \frac{1}{p}} \Vert \ueps^{\alpha} \Vert_{L^p(\oem)}& \le C,
\\
\Vert D_{\x}(\U_{\epsilon})_{\alpha\beta} \Vert_{L^p(\Sigma)} \le C \epsilon^{-1 - \frac{1}{p}} \Vert D(\ueps)_{\alpha \beta} \Vert_{L^p(\oem)} &\le C.
\end{align*}
Hence, $\U_{\epsilon}$ is bounded in $L^p(\Sigma)^2$ and $D(\U_{\epsilon})$ is bounded in $L^p(\Sigma)^{2\times 2}$. The  Korn-inequality in Lemma \ref{LemmaKornInequalities}\ref{LemmaKornInequalitiesWithL2Norm} implies
\begin{align*}
\Vert \U_{\epsilon} \Vert_{W^{1,p}(\Sigma)} \le C \left(\Vert D(\U_{\epsilon})\Vert_{L^p(\Sigma)} + \Vert \U_{\epsilon} \Vert_{L^p(\Sigma)} \right) \le C.
\end{align*}
This implies the existence of  $\U_0 \in W^{1,p}(\Sigma)^2$ such that up to a subsequence
\begin{align*}
\U_{\epsilon} &\rightharpoonup \U_0 &\mbox{ weakly in }& W^{1,p}(\Sigma)^2,
\\
\U_{\epsilon} &\rightarrow \U_0 &\mbox{ strongly in }& L^p(\Sigma)^2.
\end{align*}
Especially, we obtain ($Y = (0,1)^2$)
\begin{align*}
\U_{\epsilon} \rats \U_0 \quad \mbox{ in the two-scale sense in } \oem.
\end{align*}
Let $\phi \in C_0^{\infty}(\Sigma,C_{\#}^{\infty}(Z))$ and define $\bphi(\x,\y):= \frac12 \int_{-1}^1 \phi(\x,\y,y_3) dy_3$. We have  ($\alpha = 1,2$)
\begin{align*}
\int_{\Sigma}\int_Z \U_0(\x) \phi(\x,y) dy d\x &= \lim_{\epsilon \to 0} \foe \int_{\oem} \U_{\epsilon}(\x) \phi\bxfxe dx
\\
&= \lim_{\epsilon\to 0} \foe \int_{\oem} \frac{\ueps^{\alpha}(x)}{\epsilon} \bphi\left(\x,\frac{\x}{\epsilon}\right) dx 
\\
&=\int_{\Sigma} \int_Z \left(\tilde{u}_1^{\alpha}(\x) - y_3 \partial_{\alpha} u_0^3(\x) \right) \bphi(\x,\y) dy d\x
\\
&= \int_{\Sigma} \int_Z \tilde{u}_1^{\alpha} (\x) \phi(\x,y) dy d\x,
\end{align*}
what implies $ \tilde{u}_1^{\alpha} = \U_0^{\alpha} \in W^{1,p}(\Sigma)$.
For the regularity of $u_0^3$ we define, see again \cite{griso2008decompositions},
\begin{align*}
\Reps(\x):= \frac{3}{2\epsilon^3} \int_{-\epsilon}^{\epsilon} x_3 \begin{pmatrix}
\ueps^1 \\ \ueps^2 
\end{pmatrix}dx_3.
\end{align*}
We have ($\alpha = 1,2$)
\begin{align*}
\Vert \Reps^{\alpha} \Vert^p_{L^p(\Sigma)} &= \left(\frac{3}{2\epsilon^3}\right)^p \int_{\Sigma}\left\vert \int_{-\epsilon}^{\epsilon} x_3 \ueps^{\alpha} dx_3 \right\vert^p d\x \le  C \epsilon^{-3p + \frac{p}{p'}} \int_{\oem} \vert x_3 \vert^p \vert \ueps^{\alpha}\vert^p dx 
\\
&\le C \epsilon^{-p-1}  \Vert \ueps^{\alpha} \Vert_{L^p(\oem)}^p \le C.
\end{align*}
A similar calculation shows ($\alpha,\beta = 1,2$)
\begin{align*}
\Vert D_{\x}(\Reps)_{\alpha \beta}\Vert_{L^p(\Sigma)} \le C\epsilon^{-1 - \frac{1}{p}}\Vert D(\ueps)_{\alpha \beta} \Vert_{L^p(\oem)} \le C.
\end{align*}
Using again the 
 Korn-inequality in Lemma \ref{LemmaKornInequalities}\ref{LemmaKornInequalitiesWithL2Norm}, we obtain the boundedness of $\Reps$ in $W^{1,p}(\Sigma)^2$ and the existence of $\mathcal{R}_0 \in W^{1,p}(\Sigma)^2$, such that up to a subsequence
\begin{align*}
\Reps &\rightharpoonup \mathcal{R}_0 &\mbox{ weakly in }& W^{1,p}(\Sigma)^2,
\\
\Reps &\rightarrow \mathcal{R}_0 &\mbox{ strongly in }& L^p(\Sigma)^2,
\end{align*} 
especially, we obtain also the two-scale convergence of $\Reps$ to $\mathcal{R}_0$ in $\oem$. We get for all $\phi \in C_0^{\infty}(\Sigma,C_{\#}^{\infty}(Z))$ and  $\bphi(\x,\y)$ as above
\begin{align*}
\int_{\Sigma} \int_Z \mathcal{R}_0^{\alpha}(\x) \phi(\x,y) dy d\x &=  \lim_{\epsilon\to 0} \foe \int_{\oem} \Reps^{\alpha} (\x) \phi\bxfxe dx
\\
&= \lim_{\epsilon \to 0} \frac{3}{\epsilon} \int_{\oem} \frac{\ueps^{\alpha}(x)}{\epsilon} \frac{x_3}{\epsilon} \bphi\left(\x,\frac{\x}{\epsilon}\right) dx
\\
&= 3 \int_{\Sigma}\int_Z y_3\left(\tilde{u}_1^{\alpha} - y_3 \partial_{\alpha} u_0^3(\x) \right) \bphi(\x,\y) dy d\x
\\
&= - \int_{\Sigma} \int_Z \partial_{\alpha}u_0^3(\x) \phi(\x,y)dy d\x,
\end{align*}
and therefore $\partial_{\alpha } u_0^3 = - \mathcal{R}_0^{\alpha} \in W^{1,p}(\Sigma). $
This implies $u_0^3 \in W^{2,p}(\Sigma)$.
\end{proof}

The extension Theorem \ref{TheoremExtensionOperator} now immediately implies a two-scale compactness result for functions defined on the perforated layer $\oems$.
\begin{corollary}\label{CompactnessTwoScalePerforatedLayer}
Let $\ueps \in W^{1,p}(\oems)^3$ be a sequence with
\begin{align*}
 \Vert \ueps^3 \Vert_{L^p(\oems)} + \Vert \nabla \ueps \Vert_{L^p(\oems)} + \foe\Vert D(\ueps)\Vert_{L^p(\oems)} + \sum_{\alpha =1}^2 \foe \Vert \ueps^{\alpha} \Vert_{L^p(\oems)} \le C\epsilon^{\frac{1}{p}}.
\end{align*}
Then there exist $u_0^3 \in W^{2,p}(\Sigma)$ and $\tilde{u}_1 \in W^{1,p}(\Sigma)^3$ with $\tilde{u}_1^3 = 0$, such that up to a subsequence (for $\alpha = 1,2$)
\begin{align*}
\chi_{\oems}\ueps^3 &\rats \chi_{Z^s}u_0^3,
\\
\chi_{\oems}\frac{\ueps^{\alpha}}{\epsilon} &\rats \chi_{Z^s}\big(\tilde{u}_1^{\alpha} - y_3 \partial_{\alpha} u_0^3\big).
\end{align*}
\end{corollary}
\begin{proof}
Due to Theorem \ref{TheoremExtensionOperator} we can extend $\ueps$ to a function $E_{\epsilon}\ueps \in W^{1,p}(\oem)^3$ such that
\begin{align*}
 \Vert E_{\epsilon}\ueps^3 \Vert_{L^p(\oem)} +  \Vert \nabla E_{\epsilon}\ueps \Vert_{L^p(\oem)} +\foe\Vert D(E_{\epsilon}\ueps)\Vert_{L^p(\oem)} + \sum_{\alpha =1}^2\foe \Vert E_{\epsilon}\ueps^{\alpha} \Vert_{L^p(\oem)} \le C\epsilon^{\frac{1}{p}}.
\end{align*}
Proposition \ref{CompactnessTwoScaleThinLayer} applied to $E_{\epsilon}\ueps$ immediately implies the desired result.
\end{proof}
Until now we considered arbitrary boundary values on  the lateral boundary $\partial_D \oems$. Next we see that a zero boundary condition on this part is inherited to the limit functions $u_0^3$ and $\tilde{u}_1$.

\begin{proposition}\label{PropositionZeroBoundaryConditionCompactness}
Let $\ueps$ be as in Corollary \ref{CompactnessTwoScalePerforatedLayer} and additionally assume that $\ueps = 0$ on $\partial_D \oems$. Then the limit functions from Corollary \ref{CompactnessTwoScalePerforatedLayer} fulfill $u_0^3 \in W^{1,p}_0(\Sigma)$, $\nabla_{\x} u_0^3 \in W_0^{1,p}(\Sigma)^2$, and $\tilde{u}_1 \in W^{1,p}_0(\Sigma)^3$.
\end{proposition}
\begin{proof}
We extend the function $\ueps$ with Theorem \ref{TheoremExtensionOperator} to the whole thin layer $\oem$ and use the same notation for the extended function. We also use the same notations as in the proof of Proposition \ref{CompactnessTwoScaleThinLayer} (for the extended function).

From the trace-inequality in Lemma  \ref{TraceInequalityOuterBoundaryLemma} we obtain for $\alpha = 1,2$
\begin{align*}
\Vert \U_{\epsilon}^{\alpha} \Vert_{L^p(\partial \Sigma)} \le C \epsilon^{-1-\frac{1}{p}} \Vert \ueps^{\alpha} \Vert_{L^p(\partial \Sigma \times (-\epsilon , \epsilon))} \le \frac{C}{\epsilon} \Vert D(\ueps)\Vert_{L^p(\oems)}  \le C\epsilon^{\frac{1}{p}}.
\end{align*}
Hence, $\U_{\epsilon}^{\alpha} \rightarrow 0$ in $L^p(\partial \Sigma)$. Since $\U_{\epsilon}^{\alpha}$ converges weakly in $W^{1,p}(\Sigma)$ to $\tilde{u}_1^{\alpha}$, we also have the strong convergence in  $L^p(\partial \Sigma)$, what implies $\tilde{u}_1^{\alpha} \in W^{1,p}_0(\Sigma)$. Further, we have
\begin{align*}
\Vert \Reps^{\alpha}\Vert_{L^p(\partial \Sigma)} \le C\epsilon^{-1-\frac{1}{p}} \Vert \ueps^{\alpha} \Vert_{L^p(\partial \Sigma \times (-\epsilon,\epsilon))} \le C \epsilon^{\frac{1}{p}},
\end{align*}
what implies $\partial_{\alpha} u_0^3 \in W^{1,p}_0(\Sigma)$. In the same way we can argue with $\frac{1}{2\epsilon} \int_{-\epsilon}^{\epsilon} \ueps^3 dx_3$ and the function $u_0^3$.
\end{proof}

\begin{remark}
We emphasize that due to the Korn-inequality from Theorem \ref{KornInequalityPerforatedLayer} the assumptions of Proposition \ref{PropositionZeroBoundaryConditionCompactness} are already fulfilled for 
\begin{align*}
\Vert D(\ueps)\Vert_{L^p(\oem)} \le C \epsilon^{1 + \frac{1}{p}}.
\end{align*}
\end{remark}

\subsubsection{Convergence of the symmetric gradient}

Up to now we only considered compactness results for $\ueps^3$ and $\epsilon^{-1} \ueps^{\alpha}$ for $\alpha = 1,2$. However, for problems in continuum mechanics a crucial question is the convergence of the symmetric gradient.  For $D(\ueps)$ of order $\epsilon^{1+\frac{1}{p}}$ we know, due to Lemma \ref{TwoScaleCompactnessStandard}, that $\epsilon^{-1} D(\ueps)$ converges (up to a subsequence) in the two-scale sense. We will identify this limit in Proposition \ref{TSConvergenceSymmetricGradient} below and show
that it involves an additional function $u_2$ which corresponds to a second order term in the formal asymptotic expansion of $\ueps$. To establish the existence of $u_2$ we use a Helmholtz-decomposition based on an orthogonality condition. Therefore, our proof is restricted to the case $p=2$.

%
%
%
We start with a (periodic) Helmholtz-decomposition for symmetric matrix-valued $L^2$-functions. For the definitions of $H^2(\div,Z)$ and $L^2_{\#,\sigma}(Z)$ we refer to $\eqref{DefinitionHpDiv}$ and $\eqref{DefinitionLpPeriodicSigma}$ in the appendix. The duality pairing between $H^{-\frac12}(\partial Z)^3$ and $H^{\frac12}(\partial Z)^3$ is denoted by $\langle \cdot , \cdot \rangle_{\partial Z}$, see also Appendix \ref{SectionAppendix}. Let us define the spaces
\begin{align*}
L_s^2(Z) &:= \left\{ \psi \in L^2(Z)^{3\times 3} \, : \, \psi \mbox{ is}\mbox{ symmetric} \right\}
\\
L_{s,\#,\sigma}^2(Z)&:= L_s^2(Z) \cap L_{\#,\sigma}^2(Z)^3,
\\
G_{s,\#}(Z)&:= \left\{ \psi \in L_s^2(Z)  \, : \, \psi = D(p) \mbox{ for } p \in H^1_{\#}(Z)^3 \right\}.
\end{align*}

\begin{lemma}\label{HelmholtzDecomposition}
We have the following decomposition for the space $L_s^2(Z)$:
\begin{align*}
L_s^2(Z) = L_{s,\#,\sigma}^2(Z) \oplus G_{s,\#}(Z).
\end{align*}
\end{lemma}
\begin{proof}
Since $L_s^2(Z)$ is a Hilbert-space and $L_{s,\#,\sigma}^2(Z)$ is a closed subspace, it holds that
\begin{align*}
L_s^2(Z) = L_{s,\#,\sigma}^2(Z) \oplus L_{s,\#,\sigma}^2(Z)^{\perp}.
\end{align*}
Hence, it remains to show $G_{s,\#}(Z) = L_{s,\#,\sigma}^2(Z)^{\perp}$. The inclusion "$\subset$" is obvious, since we have for $D(p) \in G_{s,\#}(Z)$  and all $v \in L_{s,\#,\sigma}^2(Z)$
\begin{align*}
\int_Z D(p) : v dy = \int_Z \nabla p : v dy = -\int_Z p \cdot \underbrace{[\nabla \cdot v]}_{=0} dy + \langle v \cdot \nu , p \rangle_{\partial Z} = 0.
\end{align*}
Now, let $u \in L_{s,\#,\sigma}^2(Z)^{\perp}$ and therefore $ \nabla \cdot u \in H^{-1}(Z)^3$. Let $p \in H_0^1(Z)^3 \subset H_{\#}^1(Z)^3$ be the unique weak solution of 
\begin{align*}
-\nabla \cdot D(p) &= -\nabla \cdot u &\mbox{ in }& Z,
\\
p &= 0 &\mbox{ on }& \partial Z.
\end{align*}
From the Korn-inequality in Lemma \ref{LemmaKornInequalities}\ref{LemmaKornInequalitiesZeroSubsetBoundary} and the Lax-Milgram-Lemma follows the existence of a unique weak solution $p$. Especially, we have 
\begin{align*}
\nabla \cdot (u - D( p)) = 0,
\end{align*}
and therefore $u - D(p) \in H^2(\div,Z)^3$ with normal-trace in $H^{-\frac12}(\partial Z)^3$. Now, we consider the problem to find $q \in H^1_{\#}(Z)^3/\R^3$ such that
\begin{align}
\begin{aligned}\label{AuxiliaryProblemHelmholtz}
-\nabla \cdot D(q) &= 0 &\mbox{ in }& Z,
\\
-D(q)\nu &= -(u - D(p)) \nu &\mbox{ on }& S^{\pm},
\\
q \mbox{ is } Y\mbox{-periodic}&, \, \int_Z q dy = 0.
\end{aligned}
\end{align}
A weak solution of $\eqref{AuxiliaryProblemHelmholtz}$ fulfills for all $\phi \in H_{\#}^1(Z)^3$ 
\begin{align*}
\int_Z D(q) : D(\phi) dy = \langle (u - D(p))\nu , \phi \rangle_{\partial Z}.
\end{align*}
Since the only rigid-displacement with $Y$-periodic boundary condition and mean value zero is the zero-function, we obtain from the Korn-inequality in Lemma \ref{LemmaKornInequalities}\ref{LemmaKornInequalitiesSubsetIntersectionRDZero} and the Lax-Milgram lemma the existence of a unique weak solution $q$ of $\eqref{AuxiliaryProblemHelmholtz}$. From our construction we obtain
\begin{align*}
\nabla \cdot \left(u - D(p) - D(q)\right) = 0, \qquad 
\left\langle \left(u - D(p) - D(q) \right)\nu , \phi \right\rangle_{\partial Z} = 0 
\end{align*}
for all $\phi \in H_{\#}^1(Z)^3$. Hence, $u - D(p+q) \in L^2_{s,\#,\sigma}(Z)$ and since $D(p+q) \in G_{s,\#}(Z)\subset L^2_{s,\#,\sigma}(Z)^{\perp}$, we get $u - D(p+q) \in L^2_{s,\#,\sigma}(Z)^{\perp}$. Therefore we have $u = D(p+q)$ which implies the desired result.
\end{proof}
In the next Lemma we show an integration by parts formula for oscillating test functions in $H_{\#,0}^p(\div,Z)$ (for the definition see $\eqref{DefinitionHper0Div}$), i.\,e., for periodic functions with weak divergence and vanishing normal trace on $S^{\pm}$.  The crucial point is that such kind of functions can be approximated by smooth periodic functions vanishing in a neighborhood of $S^{\pm}$, which is shown in Appendix \ref{SectionAppendix}.
\begin{lemma}\label{AuxiliaryLemmaCompactnessSymGrad}
 Let $\phi \in C_0^{\infty}(\Sigma)$ and $\psi \in H_{\#,0}^{p'}(\div,Z)$. Then for all $\veps \in W^{1,p}(\oem)$ it holds the following integration by parts formula
\begin{align}\label{FormulaIntegrationParts}
\int_{\oem} \nabla \veps \cdot \psi\left(\fxe\right) \phi(\x) dx = -\int_{\oem} \veps \left[ \foe \nabla_y \cdot \psi\left(\fxe\right) \phi(\x)  +  \psi\left(\fxe\right) \cdot \nabla_{\x} \phi(\x)   \right] dx.
\end{align}
\end{lemma}
\begin{proof}
From Theorem \ref{TheoremDensitySmoothCompactPeriodicFunctionsHdiv} in the appendix we know that there exists a sequence $\psi^l \in C_{\#,0}^{\infty}(Z)^3$ with $\psi^l \rightarrow \psi $ in $H^{p'}(\div,Z)$. We consider the $Y$-periodic extension of $\psi^l$ and $\psi$. Then we have $\phi(\x)\psi^l\left(\fxe\right) \in C_0^{\infty}(\oem)^3$. By integration by parts we get
\begin{align*}
&\int_{\oem} \nabla \veps \cdot  \psi^l\left(\fxe\right) \phi(\x) dx 
= - \int_{\oem} \veps \left[ \foe \nabla_y \cdot \psi^l\left(\fxe\right) \phi(\x) +  \psi^l\left(\fxe\right) \cdot \nabla_{\x}\phi(\x) \right] dx
\\
&=- \sum_{k\in K_{\epsilon}} \epsilon^n \int_Z \veps(\epsilon (k + y)) \left[ \foe \nabla_y \cdot \psi^l(y) \phi(\epsilon(\bar{k} +\y)) +  \psi^l(y) \cdot  \nabla_{\x} \phi (\epsilon(\bar{k} + \y)) \right]  dy
\\
&\overset{l \to \infty}{\longrightarrow}- \sum_{k\in K_{\epsilon}} \epsilon^n \int_Z \veps(\epsilon (k + y)) \left[ \foe \nabla_y \cdot \psi(y) \phi(\epsilon(\bar{k} +\y)) +  \psi(y) \cdot  \nabla_{\x} \phi (\epsilon(\bar{k} + \y)) \right]  dy
\\
&=-\int_{\oem} \veps \left[ \foe \nabla_y \cdot \psi\left(\fxe\right) \phi(\x)  +  \psi\left(\fxe\right) \cdot \nabla_{\x} \phi(\x)   \right] dx.
\end{align*}
In the same way we obtain for the left-hand side
\begin{align*}
\int_{\oem} \nabla \veps \cdot  \psi^l\left(\fxe\right) \phi(\x) dx  \overset{l\to \infty}{\longrightarrow} \int_{\oem} \nabla\veps \cdot  \psi\left(\fxe\right) \phi(\x) dx.
\end{align*}
Altogether we obtain the identity $\eqref{FormulaIntegrationParts}$. 

\end{proof}

Now we are able to give the two-scale compactness result for the symmetric gradient $\epsilon^{-1} D(\ueps)$ for $p=2$:

\begin{proposition}\label{TSConvergenceSymmetricGradient}
Let $\ueps \in H^1(\oem)^3$ be a sequence with
\begin{align*}
\Vert \ueps^3 \Vert_{L^2(\oem)} + \Vert \nabla \ueps \Vert_{L^2(\oem)} + \foe\Vert D(\ueps)\Vert_{L^2(\oem)} + \sum_{\alpha =1}^2 \foe\Vert \ueps^{\alpha} \Vert_{L^2(\oem)} \le C\sqrt{\epsilon}.
\end{align*}
Let $u_0^3 \in H^2(\Sigma)$ and $\tilde{u}_1 \in H^1(\Sigma)^3$ as in Proposition \ref{CompactnessTwoScaleThinLayer}. Then there exists $u_2 \in L^2(\Sigma,H^1_{\#}(Z)/\R)^3$, such that up to a subsequence
\begin{align*}
\frac{1}{\epsilon} D(\ueps) \rats D_{\x}(\tilde{u}_1) - y_3 \nabla_{\x}^2 u_0^3 + D_y(u_2).
\end{align*}
\end{proposition}
\begin{proof}
There exists $\xi \in L^2(\Sigma \times Z)^{3\times 3}$, such that up to a subsequence
\begin{align*}
\frac{1}{\epsilon} D(\ueps) \rats \xi.
\end{align*}
Let $\phi \in C_0^{\infty}(\Sigma)$ and $\psi \in L^2_{s,\#,\sigma}(Z)$. For $\Phi (\x,y):= \phi(\x) \psi(y)$ we obtain, due to the two-scale convergence of $\epsilon^{-1} D(\ueps)$,
\begin{align*}
\lim_{\epsilon\to 0} \frac{1}{\epsilon^2} \int_{\oem} D(\ueps): \Phi\bxfxe dx = \int_{\Sigma}\int_Z \xi(\x,y) : \Phi (\x,y) dy d\x.
\end{align*}
Integrating by parts on the left-hand side gives with the integration by parts formula $\eqref{FormulaIntegrationParts}$ from Lemma \ref{AuxiliaryLemmaCompactnessSymGrad} (and using that $\psi$ is symmetric)
\begin{align}
\begin{aligned}
\label{IdentitySymGradComp}
\frac{1}{\epsilon^2} \int_{\oem} D(\ueps): \psi\left(\fxe\right) \phi(\x) dx &= \frac{1}{\epsilon^2} \int_{\oem} \nabla \ueps : \psi\left(\fxe\right) \phi(\x) dx 
\\
&= - \frac{1}{\epsilon^2} \int_{\oem} \ueps \cdot \left[ \psi\left(\fxe\right) \nabla_{\x}\phi(\x) \right] dx.
\end{aligned}
\end{align}
To pass to the limit on the right-hand side we can use the two-scale convergence of $\frac{\ueps^{\alpha}}{\epsilon}$ from Proposition \ref{CompactnessTwoScaleThinLayer} for $\alpha = 1,2$ to obtain
\begin{align*}
- \lim_{\epsilon \to 0}\sum_{\alpha = 1}^2\frac{1}{\epsilon^2}\int_{\oem} \ueps^{\alpha} &\left[\psi\left(\fxe\right) \nabla_{\x} \phi(\x) \right]_{\alpha} dx  
\\
&= -\sum_{\alpha = 1}^2 \int_{\Sigma} \int_Z (\tilde{u}_1^{\alpha}(\x) - y_3 \partial_\alpha u_0^3(\x)) \left[\psi(y)\nabla_{\x} \phi(\x)\right]_{\alpha} dy d\x
\\
&= \int_{\Sigma}\int_Z \left(D_{\x}(\tilde{u}_1) - y_3 \nabla_{\x}^2 u_0^3(\x) \right) : \psi(y) \phi(\x) dy d\x.
\end{align*}
For the term on the right-hand side in $\eqref{IdentitySymGradComp}$ including $\ueps^3$,  we use the fact that $[\psi(y)\nabla_{\x} \phi(\x)]_3 = \sum_{\alpha = 1}^2 \psi_{3\alpha}(y) \partial_{\alpha} \phi(\x)$ has mean-value zero on $Z$. In fact, we have for $\alpha = 1,2$ ($\psi_{\alpha}$ is the $\alpha$-column of $\psi$)
\begin{align}\label{MeanValuePropertyTestFunctionConvergenceSymGra}
\int_Z \psi_{3\alpha}(y) dy = \int_Z \nabla y_3 \cdot \psi_{\alpha}(y) dy = \langle \psi_{\alpha}\cdot \nu , y_3 \rangle_{H^{-\frac12}(\partial Z),H^{\frac12}(\partial Z)} = 0.
\end{align}
From the proof of Proposition \ref{AssumptionInequalityCompactnessThinLayer} we have
\begin{align*}
\nabla \ueps^3 \rats \nabla_{\x} u_0^3 + \nabla_y u_1^3(\x,y) = \nabla_{\x} u_0^3,
\end{align*}
since we have $u_1^3(\x,y) = 0$. Using the mean-value property of $[\psi(y)\nabla_{\x}\phi(\x)]_3$ we obtain with similar arguments as in $\eqref{ConvergenceTestFunctionsMeanZero}$
\begin{align*}
\lim_{\epsilon\to 0}\frac{1}{\epsilon^2} \int_{\oem} \ueps^3 \left[\psi\left(\fxe\right) \nabla_{\x}\phi(\x)\right]_3 dx = \int_{\Sigma}\int_Z u_1^3(\x,y)  [\psi(y)\nabla_{\x}\phi(\x)]_3 dy d\x = 0.
\end{align*}
Altogether we obtain
\begin{align}\label{AuxiliaryIdentityConvergenceSymmetricGradient}
\int_{\Sigma}\int_Z \left[ \xi(\x,y) - D_{\x}(\tilde{u}_1)(\x) + y_3 \nabla_{\x}^2 u_0^3(\x)\right] : \psi(y) \phi(\x) dy d\x = 0.
\end{align}
Using the Helmholtz-type-decomposition from Lemma \ref{HelmholtzDecomposition}, we obtain the existence of $u_2 \in L^2(\Sigma,H^1_{\#}(Z)/\R)^3$, such that
\begin{align*}
\xi (\x,y)  = D_{\x}(\tilde{u}_1)(\x) - y_3 \nabla_{\x}^2 u_0^3(\x) + D_y(u_2).
\end{align*} 
We emphasize that the function $u_2$ is unique up to a rigid-displacement, but the only periodic rigid-displacements are the constant functions.
\end{proof}

\begin{remark}
From the proof of Proposition \ref{TSConvergenceSymmetricGradient} we see that the result is also valid for $p\neq 2 $ (with obvious modifications on the \textit{a priori} estimate), if the Helmholtz-decomposition in Lemma \ref{HelmholtzDecomposition} is valid in $L^p$. Such results exist for functions in $L^p(Z)^n$ with periodic boundary conditions on the lateral boundary of $Z$, see \cite{sauer2018instationary}. However, in the case of symmetric $L^p$ functions this seems to be an open problem.
\end{remark}

\begin{corollary}\label{KorollarKonvergenzSymGrad}
Let $\ueps \in H^1(\oems)^3$ be a sequence with
\begin{align*}
 \Vert \ueps^3 \Vert_{L^2(\oems)} + \Vert \nabla \ueps \Vert_{L^2(\oems)} + \foe\Vert D(\ueps)\Vert_{L^2(\oems)} + \sum_{\alpha =1}^2 \foe \Vert \ueps^{\alpha} \Vert_{L^2(\oems)} \le C\sqrt{\epsilon}.
\end{align*}
Let $u_0^3 \in H^2(\Sigma)$ and $\tilde{u}_1 \in H^1(\Sigma)^3$ be as in Corollary \ref{CompactnessTwoScalePerforatedLayer}. Then there exists $u_2 \in L^2(\Sigma,H^1_{\#}(Z)/\R)^3$, such that up to a subsequence
\begin{align*}
\frac{1}{\epsilon} \chi_{\oems} D(\ueps) \rightarrow  \chi_{Z^s} \left(D_{\x}(\tilde{u}_1) - y_3 \nabla_{\x}^2 u_0^3 + D_y(u_2) \right).
\end{align*}
\end{corollary}
\begin{proof}
This is an immediate consequence of Proposition \ref{TSConvergenceSymmetricGradient} and the extension from Theorem \ref{TheoremExtensionOperator}.
\end{proof}

\begin{remark}
All the results in this section ($p=2$ for Proposition \ref{TSConvergenceSymmetricGradient} and Corollary  \ref{KorollarKonvergenzSymGrad}) remain valid for the time-dependent case  $\ueps \in L^p((0,T),W^{1,p}(\oems))^3$.
\end{remark}

\section{Homogenization of a semi-linear wave equation in a thin perforated layer}
\label{SectionApplication}

As a proof of concept scenario we consider a semi-linear elastic wave equation in the thin perforated layer $\oems \subset \R^3$ with an inhomogeneous Neumann boundary condition on the oscillating surface $\geps$. In the limit $\epsilon \to 0$ we obtain a time dependent plate equation with effective coefficients arising from the perforations in the layer, see problem $\eqref{MacroModelStrongFormulation}$. This zeroth order approximation only contains the second order time derivative of the vertical displacement, whereas the other time derivatives vanish.
 To pass to the limit $\epsilon\to 0$ we make use of the two-scale compactness results in Section \ref{SectionTwoScaleCompactness}. However, since these compactness results only have a weak character, they are not sufficient to treat the nonlinear terms. For these terms we make use of the Aubin-Lions Lemma, which is applied to the extension from Theorem \ref{TheoremExtensionOperator} of the microscopic solution after rescaling the thin layer to a layer with fixed thickness.

We consider the following microscopic model: Let $\ueps : (0,T)\times \oems \rightarrow \R^3$ be the microscopic displacement  which is described by  
\begin{subequations}\label{MicroModel}
\begin{align}
\foe \partial_{tt} \ueps - \frac{1}{\epsilon^3}\nabla \cdot (A_{\epsilon} D(\ueps)) &=  \frac{1}{\epsilon^2} f_{\epsilon}(t,x,\ueps^3) &\mbox{ in }& (0,T)\times \oems,
\\
- \frac{1}{\epsilon^3} A_{\epsilon} D(\ueps) \nu &= \foe g_{\epsilon}(x) &\mbox{ on }& (0,T)\times \geps,
\\
\ueps &= 0 &\mbox{ on }& (0,T) \times \partial_D \oems,
\\
\ueps(0) = \partial_t \ueps(0) &= 0 &\mbox{ in }& \oems.
\end{align}
\end{subequations}
For the functions $f_{\epsilon}$ and $g_{\epsilon}$ we make the following assumptions: For $(t,x,z) \in (0,T)\times \oems \times \R$ we set
\begin{align*}
f_{\epsilon} (t,x,z) = \left(
 f^1\left(t,\x,\fxe , z\right) ,  f^2\left(t,\x,\fxe,z\right) , \epsilon f^3 \left(t,\x,\fxe,z\right)
\right)^T,
\end{align*}
with $f^i \in C^0\left([0,T]\times \overline{\Sigma} \times \overline{Z^s} \times \R\right)$ for $i=1,2,3$, $Y$-periodic with respect to the third variable, uniformly Lipschitz continuous with respect to the last variable, and $\partial_t f^i, \, \partial_z f^i \in L^{\infty}((0,T)\times \Sigma \times Z^s \times \R)$.
Further 
\begin{align*}
g_{\epsilon}(x) := \left(
g^1\bxfxe , g^2 \bxfxe , \epsilon g^3 \bxfxe 
\right)^T,
\end{align*}
with $g^i \in C^0(\overline{\Sigma},C_{\#}^0(\overline{\Gamma}))$ for $i=1,2,3$. The elasticity tensor $A_{\epsilon}$ is defined by $A_{\epsilon}(x) := A\left(\fxe\right)$ with $A\in L_{\#}^{\infty}(Z^s)^{3\times 3 \times 3 \times 3}$ symmetric and coercive on the space of symmetric matrices, more precisely for $i,j,k,l=1,2,3$
\begin{align*}
A_{ijkl} &= A_{jilk} = A_{ljik},
\\
A(y) B : B &\geq c_0 \vert B\vert^2 \quad \mbox{ for almost  every } y \in Z,
\end{align*}
with $c_0 \gr 0$ and all $B \in \R^{3\times 3}$ symmetric.
We call $\ueps $ a weak solution of problem $\eqref{MicroModel}$ if 
\begin{align*}
\ueps \in L^2((0,T),H^1(\oems))^3\cap H^1((0,T),L^2(\oems))^3\cap H^2((0,T),H^1(\oems)')^3
\end{align*}
with $\ueps = 0$ on $\partial_D \oems$ and $\ueps(0) = \partial_t \ueps (0) = 0$, and for all $\phieps \in H^1(\oems)^3$ with $\phieps = 0 $ on $\partial_D \oems$ it holds almost everywhere in $(0,T)$
\begin{align}
\begin{aligned}\label{MicroProblVarForm}
\foe \langle \partial_{tt} \ueps , &\phieps \rangle_{H^1(\oems)',H^1(\oems)} + \frac{1}{\epsilon^3} \int_{\oems} A_{\epsilon} D(\ueps): D(\phieps) dx 
\\
&=  \frac{1}{\epsilon^2} \int_{\oems } f_{\epsilon}(\ueps^3) \cdot \phieps dx - \foe  \int_{\geps} g_{\epsilon} \cdot \phieps d\sigma.
\end{aligned}
\end{align}
Using the Galerkin-method and a fixed-point argument it is straightforward to show the existence of a weak solution of problem $\eqref{MicroModel}$.

\subsection*{The macroscopic model}
The aim is the derivation of a macroscopic model for $\epsilon \to 0$. We will show that the  limit of $\ueps$ in the two-scale sense is a solution of the following macro-model: 
We use the notation   (the tensors $b^{\ast}$ and $c^{\ast}$ are defined below)
\begin{align*}
\Delta_{\x} : \big(b^{\ast} D_{\x}(\tilde{u}_1) +c^{\ast} \nabla_{\x}^2 u_0^3\big) := \sum_{i,j,k,l=1}^2 \partial_{kl}\left(b^{\ast}_{ijkl} D_{\x}(\tilde{u}_1)_{ij} +  c^{\ast}_{ijkl}  \partial_{ij} u_0^3\right).
\end{align*}
Then, we are looking for a solution $(u_0^3, \tilde{u}_1)$ with $u_0^3 : (0,T)\times \Sigma \rightarrow \R$ and $\tilde{u}_1 : (0,T)\times \Sigma \rightarrow \R^2$ of 
\begin{align}
\begin{aligned}\label{MacroModelStrongFormulation}
-\nabla_{\x} \cdot \left(a^{\ast} D_{\x}(\tilde{u}_1) + b^{\ast} \nabla_{\x}^2 u_0^3 \right) &=  \bar{h}(t,\x,u_0^3) &\mbox{ in }& (0,T)\times \Sigma,
\\
\partial_{tt} u_0^3 + \Delta_{\x} : \left(b^{\ast} D_{\x}(\tilde{u}_1) + c^{\ast}\nabla_{\x}^2 u_0^3 \right) &= h^3(t,\x,u_0^3) +\nabla_{\x} \cdot \bar{H}(t,\x,u_0^3)  &\mbox{ in }& (0,T)\times \Sigma,
\\
u_0^3 = \nabla_{\x} u_0^3 \cdot \nu &= 0 &\mbox{ on }& (0,T)\times \partial \Sigma,
\\
\tilde{u}_1  &= 0 &\mbox{ on }& (0,T)\times \partial\Sigma,
\\
u_0^3(0) = \partial_t u_0^3(0) &= 0 &\mbox{ in }& \Sigma,
\end{aligned}
\end{align}
where $h= (\bar{h},h^3) = (h^1,h^2,h^3)$ and $\bar{H}= (H^1,H^2)$ are defined by ($i=1,2,3$, $\alpha = 1,2$)
\begin{align*}
h^i(t,\x,u_0^3)&:= \frac{1}{\vert Z^s\vert} \int_{Z^s} f^i(t,\x,y,u_0^3) dy - \frac{1}{\vert Z^s\vert }\int_{\Gamma}  g^i(\x,y) d\sigma_y,
\\
H^{\alpha}(t,\x,u_0^3) &:= \frac{1}{\vert Z^s\vert} \int_{Z^s} f^i(t,\x,y,u_0^3)y_3 dy - \frac{1}{\vert Z^s\vert }\int_{\Gamma}  g^i(\x,y)y_3 d\sigma_y,
\end{align*}
and the effective elasticity coefficients $a^{\ast},\, b^{\ast},\, c^{\ast} \in \R^{2\times 2 \times 2 \times 2}$ are  defined for  $\alpha,\beta,\gamma,\delta = 1,2$ by
\begin{align*}
a^{\ast}_{\alpha\beta\gamma\delta} &:=  \frac{1}{\vert Z^s \vert} \int_{Z^s} A  \left(D_y(\chi_{\alpha \beta}) + M_{\alpha\beta}\right): \left(D_y (\chi_{\gamma\delta})  + M_{\gamma\delta} \right)dy,
\\
b^{\ast}_{\alpha\beta\gamma\delta} &:=   \frac{1}{\vert Z^s \vert} \int_{Z^s} A \left(D_y(\chi_{\alpha \beta}^B)  - y_3 M_{\alpha\beta} \right) : \left(D_y (\chi_{\gamma\delta})  + M_{\gamma\delta} \right)dy,
\\
c^{\ast}_{\alpha\beta\gamma\delta} &:=   \frac{1}{\vert Z^s \vert} \int_{Z^s} A  \left(D_y(\chi_{\alpha \beta}^B)  - y_3 M_{\alpha\beta} \right): \left(D_y (\chi_{\gamma\delta}^B)  -y_3 M_{\gamma\delta}\right)dy,
\end{align*}
where the cell solutions $\chi_{\alpha \beta }$ and $\chi_{\alpha \beta}^B$, and the matrices $M_{\alpha \beta}$ are defined in $\eqref{CellProblemStandard}$, $\eqref{CellProblemHesse}$, and $\eqref{DefinitionBasisMatrices}$ below.
\\
We say that $(u_0^3,\tilde{u}_1)$ is a weak solution of problem $\eqref{MacroModelStrongFormulation}$ if
\begin{align*}
u_0^3 &\in L^2((0,T),H^2_0(\Sigma))\cap H^1((0,T),L^2(\Sigma))\cap H^2((0,T),H^{-2}(\Sigma)),
\\
\tilde{u}_1 = (\tilde{u}_1^1 , \tilde{u}_1^2 ,0 )^T &\in L^2((0,T),H^1_0(\Sigma))^3
\end{align*}
and for all all $V \in H_0^2(\Sigma)$ and $U=(\bar{U},0) \in H_0^1(\Sigma)^3$  it holds almost everywhere in $(0,T)$ that
\begin{align}
\begin{aligned}\label{MacroModelVarForm}
\langle \partial_{tt} & u_0^3, V \rangle_{H^{-2}(\Sigma),H_0^2(\Sigma)} 
\\
&+  \int_{\Sigma} a^{\ast} D_{\x}(\tilde{u}_1) : D_{\x}(\bar{U}) + b^{\ast} \nabla_{\x}^2 u_0^3 : D_{\x}(\bar{U}) + b^{\ast} D_{\x}(\tilde{u}_1) : \nabla_{\x}^2 V + c^{\ast} \nabla_{\x}^2 u_0^3 : \nabla_{\x}^2 V d\x
\\
=&   \int_{\Sigma} \bar{h}(t,\x,u_0^3) \cdot  \bar{U} d\x + \int_{\Sigma} h^3(t,\x,u_0^3) V d\x - \int_{\Sigma} \bar{H} \cdot \nabla_{\x} V d\x.
\end{aligned}
\end{align}

\subsection*{\textit{A priori} estimates}
The following \textit{a priori} estimates will be used to derive the macroscopic model based on the the two-scale compactness results from Section \ref{SectionTwoScaleCompactness}.

%

\begin{lemma}
\label{AprioriEstimatesLemma}
A weak solution $\ueps$ of the micro-model $\eqref{MicroModel}$ fulfills 
\begin{align*}
\frac{1}{\sqrt{\epsilon}} \Vert \partial_t \ueps \Vert_{L^{\infty}((0,T), L^2(\oems))}  + \frac{1}{\epsilon^{\frac32}} \Vert D(\ueps)\Vert_{L^{\infty}((0,T),L^2(\oems))} \le C
\end{align*}
\end{lemma}
\begin{proof}
We test the equation $\eqref{MicroProblVarForm}$ with $\partial_t \ueps $ to obtain  
\begin{align*}
\frac{1}{2\epsilon} \frac{d}{dt} \Vert \partial_t \ueps &\Vert_{L^2(\oems)}^2 + \frac{1}{\epsilon^3} \frac{d}{dt} \int_{\oems} A_{\epsilon} D(\ueps) : D(\ueps) dx  
\\
&= \frac{1}{\epsilon^2} \int_{\oems } f_{\epsilon}(\ueps^3) \cdot \partial_t \ueps  dx - \foe  \int_{\geps} g_{\epsilon} \cdot \partial_t \ueps  d\sigma
\end{align*}
Integration in time and using the coercivity of $A_{\epsilon}$ we obtain for almost every $t \in (0,T)$
\begin{align*}
\frac{1}{2\epsilon} \Vert &\partial_t \ueps(t) \Vert^2_{L^2(\oems)} + \frac{c_0}{2\epsilon^{3}} \Vert D(\ueps)(t)\Vert_{L^2(\oems)} ^2
\\
\le& \sum_{\alpha = 1}^2 \left[ \frac{1}{\epsilon^2} \int_0^t \int_{\oems} f^{\alpha}\left(s,\x,\fxe , \ueps^3\right) \partial_t \ueps^{\alpha} dx ds - \foe  \int_{\geps} g^{\alpha}\bxfxe \ueps^{\alpha}(t,x) d\sigma\right] 
\\
&+  \foe \int_0^t \int_{\oems} f^3\left(s,\x,\fxe,\ueps^3\right) \partial_t \ueps^3 dx  ds -   \int_{\geps} g^3\bxfxe \ueps^3(t,x)d\sigma.
\end{align*}
To estimate the boundary terms we make use of the well-known trace inequality
\begin{align*}
\sqrt{\epsilon}\Vert \veps \Vert_{L^2(\geps)} \le  C \left(\Vert \veps\Vert_{L^2(\oems)} + \epsilon \Vert \nabla \veps \Vert_{L^2(\oems)} \right) \quad \mbox{for all } \veps \in H^1(\oems).
\end{align*}
For $\alpha = 1,2$ we obtain with the Korn-inequality from Theorem \ref{KornInequalityPerforatedLayer}
\begin{align*}
\left\vert  \foe \int_{\geps} g^{\alpha}\bxfxe \ueps^{\alpha}(t,x) d\sigma  \right\vert &\le \frac{C}{\epsilon} \Vert \ueps^{\alpha} (t)\Vert_{L^1(\geps)} \le  \frac{C}{\epsilon} \Vert \ueps^{\alpha}(t)\Vert_{L^2(\geps)}
\\
&\le C \left(\frac{1}{\epsilon^{\frac{3}{2}}} \Vert \ueps^{\alpha}(t) \Vert_{L^2(\oems)} + \frac{1}{\sqrt{\epsilon}} \Vert \nabla \ueps^{\alpha}(t) \Vert_{L^2(\oems)} \right)
\\
&\le  \frac{C}{\epsilon^{\frac32}} \Vert D(\ueps)(t)\Vert_{L^2(\oems)}
\\
&\le C(\theta) + \frac{\theta}{\epsilon^3} \Vert D(\ueps)(t)\Vert_{L^2(\oems)}^2,
\end{align*}
for arbitrary $\theta \gr 0$.
Further, with similar arguments we get
\begin{align*}
\left\vert \int_{\geps} g^3\bxfxe \ueps^3(t,x)d\sigma \right\vert  &\le C\left(\frac{1}{\sqrt{\epsilon}} \Vert \ueps^3(t)\Vert_{L^2(\oems)} + \sqrt{\epsilon} \Vert \nabla \ueps^3(t)\Vert_{L^2(\oems)}\right)
\\
&\le C(\theta) + \frac{\theta}{\epsilon^3} \Vert D(\ueps)(t)\Vert_{L^2(\oems)}^2.
\end{align*}
For the terms in the bulk domains we get for $\alpha = 1,2$ using an integration by parts in time
\begin{align*}
\frac{1}{\epsilon^2} \int_0^t \int_{\oems} f^{\alpha}\left(s,\x,\fxe , \ueps^3\right) \partial_t \ueps^{\alpha} dx ds =&  - \frac{1}{\epsilon^2} \int_0^t \int_{\oems} \partial_t f^{\alpha} \left(s,\x,\fxe , \ueps^3\right)  \ueps^{\alpha} dx ds
\\
&- \frac{1}{\epsilon^2} \int_0^t \int_{\oems} \partial_z f^{\alpha} \left(s,\x,\fxe , \ueps^3\right) \partial_t \ueps^3 \ueps^{\alpha} dx ds
\\
&+ \frac{1}{\epsilon^2} \int_{\oems} f^{\alpha} \left(t,\x,\fxe , \ueps^3\right) \ueps^{\alpha}(t,x) dx  
\\
&=: A_{\epsilon}^1 +  A_{\epsilon}^2 + A_{\epsilon}^3.
\end{align*}
We have with the Korn-inequality in Theorem \ref{KornInequalityPerforatedLayer}
\begin{align*}
\vert A_{\epsilon}^1 \vert &\le \frac{C}{\epsilon^2} \Vert \ueps^{\alpha} \Vert_{L^1((0,t)\times \oems)} \le \frac{C}{\epsilon^{\frac32}} \Vert \ueps^{\alpha} \Vert_{L^2((0,t)\times \oems)} 
\\
&\le  \frac{C}{\epsilon^{\frac32}} \Vert D(\ueps)\Vert_{L^2((0,t)\times \oems)}
\le C  + \frac{C}{\epsilon^3} \Vert D(\ueps)\Vert_{L^2((0,t)\times \oems)}^2. 
\end{align*}
In a similar way we get
\begin{align*}
\vert A_{\epsilon}^2 \vert 
\le \frac{C}{\epsilon} \Vert \partial_t \ueps^3\Vert_{L^2((0,t)\times \oems)}^2 + \frac{C}{\epsilon^3} \Vert D(\ueps)\Vert_{L^2((0,t)\times \oems)}^2
\end{align*}
Using  again the Korn-inequality, we get for arbitrary $\theta \gr 0$
\begin{align*}
\vert A_{\epsilon}^3 \vert &\le \frac{C}{\epsilon^2} \Vert \ueps^{\alpha}(t)\Vert_{L^1(\oems)} \le \frac{C}{\epsilon^{\frac32}}\Vert \ueps^{\alpha}(t)\Vert_{L^2(\oems)}
\\
&\le C(\theta) + \frac{\theta}{\epsilon^3} \Vert \ueps^{\alpha}(t)\Vert_{L^2(\oems)} 
\le C(\theta) + \frac{\theta}{\epsilon^3} \Vert D(\ueps)(t)\Vert_{L^2(\oems)} 
\end{align*}
Hence, we obtain
\begin{align*}
\bigg\vert \frac{1}{\epsilon^2}& \int_0^t  \int_{\oems} f^{\alpha}\left(s,\x,\fxe , \ueps^3\right) \partial_t \ueps^{\alpha} dx ds \bigg\vert  
\\
\le&  \frac{C}{\epsilon} \Vert \partial_t \ueps^3\Vert_{L^2((0,t)\times \oems)}^2 + \frac{C}{\epsilon^3} \Vert D(\ueps)\Vert_{L^2((0,t)\times \oems)}^2 
+ C(\theta) + \frac{\theta}{\epsilon^3} \Vert D(\ueps)(t)\Vert_{L^2(\oems)}
\end{align*}
In a similar way we obtain
\begin{align*}
\bigg\vert  \frac{1}{\epsilon} & \int_0^t  \int_{\oems} f^{3}\left(s,\x,\fxe , \ueps^3\right) \partial_t \ueps^{3} dx ds \bigg\vert  
\\
\le&  \frac{C}{\epsilon} \Vert \partial_t \ueps^3\Vert_{L^2((0,t)\times \oems)}^2 + \frac{C}{\epsilon^3} \Vert D(\ueps)\Vert_{L^2((0,t)\times \oems)}^2 
+ C(\theta) + \frac{\theta}{\epsilon^3} \Vert D(\ueps)(t)\Vert_{L^2(\oems)}
\end{align*}
Altogether, we obtain for $\theta $ small enough
\begin{align*}
\foe \Vert \partial_t \ueps(t)& \Vert_{L^2(\oems)}^2 + \frac{1}{\epsilon^3} \Vert D(\ueps)(t)\Vert_{L^2(\oems)}^2 
\\
&\le C \left(1 + \foe \Vert \partial_t \ueps \Vert_{L^2((0,t)\times \oems)}^2 + \frac{1}{\epsilon^3} \Vert D(\ueps)\Vert_{L^2((0,t)\times \oems)}^2 \right).
\end{align*}
The Gronwall-inequality implies
\begin{align*}
\frac{1}{\sqrt{\epsilon}} \Vert \partial_t \ueps \Vert_{L^{\infty}((0,T), L^2(\oems))} + \frac{1}{\epsilon^{\frac32}} \Vert D(\ueps)\Vert_{L^{\infty}((0,T),L^2(\oems))} \le C.
\end{align*}
\end{proof}

\subsection*{Derivation of the macroscopic model}
Based on the \textit{a priori} estimates in Lemma \ref{AprioriEstimatesLemma} we want to pass to the limit $\epsilon \to 0$ in the microscopic equation $\eqref{MicroProblVarForm}$. To cope with the nonlinear term we need a strong two-scale compactness argument, which is obtained from the extension Theorem \ref{TheoremExtensionOperator} and the Aubin-Lions lemma. We start with the convergence result for the micro-solutions. 

\begin{proposition}\label{ConvergenceMicroSolutionProp}
Let $\ueps$ be a sequence of weak solutions of the micro-model $\eqref{MicroModel}$. Then there exists $u_0^3 \in L^2((0,T),H^2_0(\Sigma))$ and $\tilde{u}_1 \in L^2((0,T), H^1_0(\Sigma)^3)$ with $\tilde{u}_1^3 = 0$, and $u_2 \in L^2(\Sigma,H_{\#}^1(Z)/\R)^3$, such that up to a subsequence ($\alpha = 1,2$)
\begin{align*}
\chi_{\oems}\ueps^3 &\rats \chi_{Z^s}u_0^3 \qquad \mbox{ strongly in the two-scale sense},
\\
\chi_{\oems}\frac{\ueps^{\alpha}}{\epsilon} &\rats \chi_{Z^s}\big(\tilde{u}_1^{\alpha} - y_3 \partial_{\alpha} u_0^3\big),
\\
\frac{1}{\epsilon} \chi_{\oems} D(\ueps) &\rats \chi_{Z^s} \left(D_{\x}(\tilde{u}_1) - y_3 \nabla_{\x}^2 u_0^3 + D_y(u_2)\right),
\\
\chi_{\oems} \partial_t \ueps^3 &\rats \chi_{Z^s} \partial_t u_0^3.
\end{align*}

\end{proposition}
\begin{proof}
The weak two-scale convergences are a direct consequence of the a priori estimates in Lemma \ref{AprioriEstimatesLemma} and Theorem \ref{MainTheoremTwoScaleConvergence}. Hence, we only have to check the strong two-scale convergence of $\chi_{\oems}\ueps^3$. First of all, we extend $\ueps$ by Theorem \ref{TheoremExtensionOperator} to the whole layer $\oem$ and use the same notation for this extension. Further, we define as in Section \ref{SectionKornInequality} the rescaled function
\begin{align*}
\tueps^3(t,x):= \ueps^3(t,\x,\epsilon x_3) \qquad \mbox{for } (t,x)\in (0,T)\times \Omega_1^M.
\end{align*}
Lemma \ref{AprioriEstimatesLemma} and the Korn-inequality in Theorem \ref{KornInequalityPerforatedLayer} imply
\begin{align*}
\Vert \partial_t \tueps^3 \Vert_{L^{\infty}((0,T),L^2(\Omega_1^M))}  + \Vert  \tueps^3 \Vert_{L^{\infty}((0,T),H^1(\Omega_1^M))} \le C.
\end{align*}
The Aubin-Lions-Lemma, see \cite{Lions}, implies the existence of $v_0^3 \in L^2((0,T),H^1(\Omega_1^M)) \cap H^1((0,T),L^2(\Omega_1^M))$, such that up to a subsequence 
\begin{align*}
\tueps^3 \rightarrow v_0^3 \qquad \mbox{in } L^2((0,T)\times \Omega_1^M).
\end{align*}
Since the extension $\ueps^3$ converges in the two-scale sense to $u_0^3$ (see also Corollary \ref{CompactnessTwoScalePerforatedLayer} for more details), it is easy to check that $v_0^3 = u_0^3$. We get
\begin{align*}
\frac{1}{\sqrt{\epsilon}} \Vert \ueps^3 - u_0^3 \Vert_{L^2((0,T)\times \oems)} \le \Vert \tueps^3 - u_0^3 \Vert_{L^2((0,T)\times \Omega_1^M)} \overset{\epsilon\to 0 }{\longrightarrow} 0.
\end{align*}
This is the strong two-scale convergence of $\chi_{\oems}\ueps^3$, which finishes the proof.
\end{proof}

\begin{corollary}\label{ConvergenceNonlinearFunction}
Up to a subsequence it holds that
\begin{align*}
f^i\left(t,\x,\fxe, \ueps^3\right) \rats f^i(t,\x,y,u_0^3).
\end{align*}
\end{corollary}
\begin{proof}
The claim follows from the strong two-scale convergence of $\ueps^3$, see for example \cite[Lemma 3.6]{GahnNeussRaduSingularLimit2021}.
\end{proof}

We define the symmetric matrices $M_{ij} \in \R^{3\times 3}$ for $i,j=1,2,3$ by
\begin{align}\label{DefinitionBasisMatrices}
M_{ij} = \frac{e_i \otimes e_j}{2} + \frac{e_j \otimes e_i}{2}.
\end{align}
Further, we define  $\chi_{ij} \in H^1_{\#}(Z^s)^3$ as the solutions of the cell problems 
\begin{align}
\begin{aligned}\label{CellProblemStandard}
-\nabla_y \cdot (A (D_y(\chi_{ij}) + M_{ij})) &= 0 &\mbox{ in }& Z^s,
\\
-A(D_y(\chi_{ij}) + M_{ij} ) \nu &= 0 &\mbox{ on }& \Gamma,
\\
\chi_{ij} \mbox{ is } Y\mbox{-periodic, } & \int_{Z^f} \chi_{ij} dy = 0.
\end{aligned}
\end{align}
Due to the Korn-inequality, this problem has a unique weak solution. We emphasize again that the only rigid-displacements on $Z^f$, which are $Y$-periodic, are constants.

Additionally, we define $\chi_{ij}^B \in H^1_{\#}(Z^s)^3$ as the solutions of the cell problems
\begin{align}
\begin{aligned}\label{CellProblemHesse}
-\nabla_y \cdot \left( A(D_y(\chi_{ij}^B) - y_3 M_{ij})\right) &= 0 &\mbox{ in }& Z^s,
\\
-A(D_y(\chi_{ij}^B) - y_3 M_{ij})\nu &= 0 &\mbox{ on }& \Gamma,
\\
\chi_{ij}^B \mbox{ is } Y\mbox{-periodic, } & \int_{Z^f} \chi_{ij}^B dy = 0.
\end{aligned}
\end{align}
In the same way as above we obtain the existence of a unique weak solution.

\begin{proposition}\label{RepresentationU2}
The limit function $u_2$  fulfills
\begin{align*}
u_2(t,\x,y) = \sum_{i,j=1}^2 \left[ D_{\x}(\tilde{u}_1)_{ij} (t,\x) \chi_{ij}(y) +  \partial_{ij}u_0^3(t,\x) \chi_{ij}^B(y) \right],
\end{align*}
where the cell solutions $\chi_{ij}$ and $\chi_{ij}^B$ are defined in $\eqref{CellProblemStandard}$ and $\eqref{CellProblemHesse}$.
\end{proposition}
\begin{proof}
Let $\phi \in C_0^{\infty}((0,T)\times \Sigma , C_{\#}^{\infty}(\overline{Z^s}))^3$  and choose in $\eqref{MicroProblVarForm}$ the test function $\phieps(t,x):= \epsilon^2 \phi\left(t,\x,\fxe\right)$. Integration in time and integration by parts together with the initial condition $\partial_t \ueps (0) = 0$ imply
\begin{align*}
- \epsilon \int_0^T &\int_{\oems} \partial_t \ueps \cdot \partial_t \phi \left(t,\x,\fxe\right) dx dt 
\\
+& \frac{1}{\epsilon} \int_0^T \int_{\oems} A_{\epsilon} \frac{D(\ueps)}{\epsilon} : \left[\epsilon D_{\x}(\phi) \left(t,\x,\fxe\right) + D_y(\phi) \left(t,\x,\fxe\right)\right] dx dt
\\
&= \int_0^T \int_{\oems} f_{\epsilon} (\ueps^3) \cdot \phi\left(t,\x,\fxe\right) dx dt - \epsilon \int_0^T \int_{\geps} g_{\epsilon} \cdot \phi\left(t,\x,\fxe\right) dx dt. 
\end{align*}
It is easy to check that all the terms except the one including $D_y(\phi)$ are of order $\epsilon$. Hence, using the convergence of $\frac{D(\ueps)}{\epsilon}$ from Proposition \ref{ConvergenceMicroSolutionProp} we obtain for $\epsilon \to 0$
\begin{align*}
0 = \int_0^T \int_{\Sigma} \int_{Z^s} A \left[D_{\x}(\tilde{u}_1)  - y_3 \nabla_{\x}^2 u_0^3 + D_y(u_2) \right] : D_y(\phi) dy d\x dt.
\end{align*}
For given $(\tilde{u}_1,u_0^3)$ this problem has a unique solution $u_2$, due to the Korn-inequality and the Lax-Milgram-Lemma. An elemental calculation gives the desired result.
\end{proof}

\begin{theorem}
The limit function $(u_0^3,\tilde{u}_1)$ of the microscopic solutions $\ueps$ from Proposition \ref{ConvergenceMicroSolutionProp} is a weak solution of the macroscopic model $\eqref{MacroModelStrongFormulation}$.
\end{theorem}
\begin{proof}
In the variational equation $\eqref{MicroProblVarForm}$ we choose test functions of the form
\begin{align*}
\phieps(t,x):= \psi(t) \left[ \begin{pmatrix}
0 \\ 0 \\ V(\x)
\end{pmatrix} 
+\epsilon \left( \begin{pmatrix}
U_1(\x) \\ U_2(\x) \\ 0 
\end{pmatrix}
- \frac{x_3}{\epsilon} \begin{pmatrix}
\partial_1 V (\x) \\ \partial_2 V(\x) \\ 0 
\end{pmatrix}
\right) \right],
\end{align*}
with $\psi \in C_0^{\infty}([0,T))$, $V \in H_0^2(\Sigma)$, and $\bar{U}:= (U_1,U_2)\in H_0^1(\Sigma)^2$.
In the following we consider $\nabla_{\x} V(\x)$ as both a vector in $\R^2$ and $\R^3$ with respect to the trivial embedding. In the same way we proceed with $\nabla_{\x}^2 V$. Integration with respect to time  in $\eqref{MacroModelStrongFormulation}$ and an integration by parts give (with $U:= (\bar{U},0)$ and the initial condition $\partial_t \ueps(0) = 0$)
\begin{align}
\begin{aligned}
\label{AuxiliaryEquationDerivationMacroModel}
-\foe &\int_0^T \int_{\oems} \partial_t \ueps^3 V \psi'+ \epsilon \partial_t \ueps \cdot \left(U - \frac{x_3}{\epsilon} \nabla_{\x} V \right)\psi' dxdt 
\\
&+ \frac{1}{\epsilon^3} \int_0^T \int_{\oems} A_{\epsilon}D(\ueps): \epsilon \left(D_{\x}(U) - \frac{x_3}{\epsilon} \nabla_{\x}^2 V \right) \psi dx dt
\\
=& \sum_{\alpha=1}^2 \bigg[ \frac{1}{\epsilon} \int_0^T \int_{\oems} f^{\alpha} \left(t,\x,\fxe,\ueps^3\right)  \left( U_{\alpha} - \frac{x_3}{\epsilon} \partial_{\alpha} V \right) \psi dx dt 
\\
&- \int_0^T \int_{\geps} g^{\alpha}\bxfxe  \left( U_{\alpha} - \frac{x_3}{\epsilon} \partial_{\alpha} V \right) \psi d\sigma dt  \bigg]
\\
&+ \foe \int_0^T \int_{\oems} f^3\left(t,\x,\fxe,\ueps^3\right) V \psi dx dt - \int_0^T \int_{\geps} g^3 \bxfxe V  \psi d\sigma dt.
\end{aligned}
\end{align}
The convergence results in Proposition \ref{ConvergenceMicroSolutionProp} and Corollary \ref{ConvergenceNonlinearFunction} imply for $\epsilon \to 0$ (for the boundary term we refer to \cite[Lemma 4.3]{BhattacharyaGahnNeussRadu})
\begin{align}
\begin{aligned}
\label{AuxiliaryMacroEquationDerivationMacroModel}
- \vert Z^s\vert & \int_0^T \int_{\Sigma}  \partial_t u_0^3 V \psi' d\x dt 
\\
+& \int_0^T \int_{\Sigma} \int_{Z^s} A \left[D_{\x}(\tilde{u}_1) - y_3 \nabla_{\x}^2 u_0^3 + D_y(u_2) \right] : \left[ D_{\x} (U) - y_3 \nabla_{\x}^2 V \right] \psi dy d\x dt 
\\
=& \sum_{\alpha =1}^2 \bigg[ \int_0^T \int_{\Sigma} \int_{Z^s} f^{\alpha}(t,\x,y,u_0^3) (U_{\alpha} - y_3 \partial_{\alpha } V ) \psi dy d\x dt 
\\
&- \int_0^T \int_{\Sigma}\int_{\Gamma} g^{\alpha} (\x,y) (U_{\alpha} - y_3 \partial_{\alpha} V ) \psi d\sigma_y d\x dt  \bigg]
\\
&+ \int_0^T \int_{\Sigma} \int_{Z^s} f^3 (t,\x,y,u_0^3) V \psi dy d\x dt - \int_0^T \int_{\Sigma}\int_{\Gamma} g^3 (\x,y)  V \psi d\sigma_y d\x dt.
\end{aligned}
\end{align}
Choosing $U= 0$ we immediately obtain $\partial_{tt} u_0^3 \in L^2((0,T), H^{-2}(\Sigma))$ and $\partial_t u_0^3(0) = 0$. An additional integration by parts in time in equation $\eqref{AuxiliaryEquationDerivationMacroModel}$ and passing to the limit $\epsilon \to 0$ implies together with $\eqref{AuxiliaryMacroEquationDerivationMacroModel}$
\begin{align*}
-\int_0^T  \int_{\Sigma} \partial_t u_0^3 V \psi' d\x dt = \int_0^T \int_{\Sigma} u_0^3 V \psi'' d\x dt 
\end{align*}
for all $\psi \in C_0^{\infty}([0,T)) $ and $ V \in H^2_0(\Sigma)$. This implies $u_0^3(0) = 0$.
 Now, using the representation of $u_2$ from Proposition 
\ref{RepresentationU2}, an elemental calculation shows
\begin{align*}
\int_{\Sigma} &\int_{Z^s} A \left[D_{\x}(\tilde{u}_1) - y_3 \nabla_{\x}^2 V + D_y(u_2) \right] : \left[ D_{\x} (U) - y_3 \nabla_{\x}^2 V \right] dy d\x 
\\
&=  \int_{\Sigma} a^{\ast} D_{\x}(\tilde{u}_1) : D_{\x}(\bar{U}) + b^{\ast} \nabla_{\x}^2 u_0^3 : D_{\x}(\bar{U}) + b^{\ast} D_{\x}(\tilde{u}_1) : \nabla_{\x}^2 V + c^{\ast} \nabla_{\x}^2 u_0^3 : \nabla_{\x}^2 V d\x.
\end{align*}
With a  density argument in $\eqref{AuxiliaryMacroEquationDerivationMacroModel}$ we obtain that $(u_0^3,\tilde{u}_1)$ solves $\eqref{MacroModelVarForm}$ and therefore is a weak solution of $\eqref{MicroModel}$. 

\end{proof}

\section{Conclusion}
\label{SectionConclusion}

We provide a variety of two-scale tools for the treatment of problems in periodic homogenization and dimension reduction for thin perforated  layers, with a special focus on applications in continuum mechanics like linear elasticity and the Navier-Stokes equations. Since our results are mainly based on simple \textit{a priori} estimates, they can be easily applied to more complex problems, as they occur for example in many applications in biosciences. In such problems one often has to deal with fluid-structure interaction and additional transport of species, leading to nonlinear coupled problems. Homogenization and dimension reduction results for such complex models including thin porous structures are very rare in the literature and our results should be an important contribution  for the treatment of such problems. In particular, our extension operator is a helpful tool when dealing with nonlinearities. Further, for the treatment of nonlinear coupling conditions $L^p$-convergence in the two-scale sense are quite important. Hence, a generalization of the compactness result for the symmetric gradient from Proposition \ref{TSConvergenceSymmetricGradient} is necessary. For this, to avoid the use of a periodic Helmholtz-decomposition for symmetric matrices in $L^p$, it might be helpful to take into account decomposition arguments from \cite{griso2008decompositions}. This is part of our ongoing work.

\begin{appendix}

\section{A density result for periodic functions with weak divergence}
\label{SectionAppendix}
The aim of this section is to prove density results for $L^p$-spaces with weak divergence and different types of  boundary conditions. More precisely, for a bounded Lipschitz domain $\Omega \subset \R^n$ and $p \in (1,\infty)$ we define 
\begin{align}\label{DefinitionHpDiv}
H^p(\div,\Omega):= \left\{ u \in L^p(\Omega)^n \, : \, \nabla \cdot u \in L^p(\Omega) \right\},
\end{align}
together with the norm
\begin{align*}
\Vert u \Vert_{H^p(\div,\Omega)}^p := \Vert u \Vert_{L^p(\Omega)}^p + \Vert \nabla \cdot u \Vert_{L^p(\Omega)}^p.
\end{align*}
It is well-known that the smooth functions $C^{\infty}(\overline{\Omega})^n$ are dense in $H^p(\div,\Omega)$, see \cite[Theorem III.2.1]{Galdi}. Further, for $u \in H^p(\div,\Omega)$ the normal trace $u\cdot \nu$ is an element in $W^{-\frac{1}{p},p}(\partial \Omega) := \left(W^{\frac{1}{p},p'}(\partial \Omega)\right)^{\prime}$ and a generalized divergence theorem is valid \cite[Theorem III.2.2]{Galdi}. For the duality pairing between $W^{-\frac{1}{p},p}(\partial \Omega) $ and $W^{\frac{1}{p},p'}(\partial \Omega)$ we use the short notation
\begin{align*}
\langle \cdot ,\cdot \rangle_{\partial \Omega} := \langle \cdot ,\cdot \rangle_{W^{-\frac{1}{p},p}(\partial \Omega),W^{\frac{1}{p},p'}(\partial \Omega)}.
\end{align*}
Then we have the following divergence formula for all $u \in H^p(\div,\Omega)$ and $\phi \in W^{1,p'}(\Omega)$
\begin{align*}
\langle u\cdot \nu , \phi \rangle_{\partial \Omega} = \int_{\Omega} \nabla \cdot u \phi + u \cdot \nabla \phi dx.
\end{align*}
In a first step we consider a cube $Y$, introduce periodic functions in $H^p(\div,Y)$, and prove a density result for smooth periodic functions. In a second step we extend these results to the cylinder $Z$ with periodic boundary on the lateral part, and a zero normal trace on the top and bottom of $Z$.

\subsection{The full periodic cell $Y$}
We start to consider the case of a cube and functions with periodic boundary conditions on the whole boundary.
To simplify the notations we define for this section (in contrast to the rest of the paper) the unit cube $Y = \left(-\frac12,\frac12\right)^n$.
We define the space of periodic functions in $H^p(\div,Y)$ by
\begin{align*}
H_{\#}^p(\div,Y):= \left\{ u \in H^p(\div,Y) \, : \, \langle u \cdot \nu , \phi \rangle_{\partial Y} = 0 \, \forall \phi \in W_{\#}^{1,p'}( Y) \right\}.
\end{align*}
We also introduce the space of smooth, solenoidal, and periodic functions
\begin{align*}
C_{\#,\sigma}^{\infty}(Y)^n:=\left\{ u \in C_{\#}^{\infty}(Y)^n \, : \, \nabla \cdot u = 0\right\},
\end{align*}
and the space
\begin{align}\label{DefinitionLpPeriodicSigma}
L^p_{\#,\sigma}(Y):= \left\{ u \in H^p_{\#}(\div,Y) \, : \, \nabla \cdot u = 0 \right\}.
\end{align}
We emphasize that on $L^p_{\#,\sigma}$ the $L^p(Y)$-norm is equal to the $H^p(\div,Y)$-norm.
The aim of this section is to prove the following result:
\begin{theorem}\label{TheoremDensity}
We have that
\begin{enumerate}
[label = (\roman*)]
\item  $C_{\#}^{\infty}(Y)^n$ is dense in $H_{\#}^p(\div,Y)$,
\item $C_{\#,\sigma}^{\infty}(Y)^n $ is dense in $L^p_{\#,\sigma}(Y)$.
\end{enumerate}
\end{theorem}
The principal idea of the proof is to use a convolution with a  periodic Dirac-sequence as kernel, which preserves the standard approximation properties of the convolution, and gives a commuting property with the weak derivatives for periodic functions.

In the following we denote by $\vert \cdot \vert$ the maximum norm on $\R^n$. Further, for $0 < \epsilon \ll 1$ we define 
\begin{align*}
Y_{\epsilon}:= \{y \in Y \, : \, dist(y,\partial Y) > \epsilon \} = \left\{ y \in Y \, : \, \vert y \vert <  \frac12 - \epsilon\right\}.
\end{align*}
Here the distance is taken with respect to the maximum norm $\vert \cdot \vert$. 

Let $\phi \in C_0^{\infty}(B_1(0))$ with $\phi \geq 0 $ and $\int_{B_1(0)} \phi dy = 1$. We define the Dirac-sequence 
\begin{align}\label{DiracSequence}
\phi_{\epsilon}(y):= \epsilon^{-n} \phi\left(\frac{y}{\epsilon}\right) \in C_0^{\infty}(B_{\epsilon}(0)).
\end{align}
We extend the functions $\phi_{\epsilon}$ $Y$-periodically to the whole $\R^n$. For $f \in L^p(Y)$ we define the convolution
\begin{align*}
f_{\epsilon}(x):= \int_Y \phi_{\epsilon}(x-y) f(y)dy \qquad\mbox{for } x \in Y.
\end{align*}
For vector valued functions  the convolution  is defined for every component.
We have $f_{\epsilon} \in C^{\infty}(Y)$ and as in the case of non-periodic convolution kernel, see for example \cite{WlokaEnglisch}, it holds $f_{\epsilon} \in L^p(Y)$ with 
\begin{align*}
\Vert f_{\epsilon} \Vert_{L^p(Y)} \le \Vert f\Vert_{L^p(Y)}.
\end{align*}
From the convolution theory it is well-known that $f_{\epsilon}$ is an approximation of $f$ in $L^p(Y)$. However, since this result is quite standard for Dirac-sequences with compact support, we will only sketch the proof for our setting and point out the differences.
\begin{lemma}\label{LemmaApproximation}
For $f \in L^p(Y)$ we have $f_{\epsilon}\rightarrow f  $ in $ L^p(Y) $ for $\epsilon \to 0$.
\end{lemma}
\begin{proof}
From the triangle inequality we get
\begin{align*}
\Vert f_{\epsilon} - f \Vert_{L^p(Y)}  
\le \Vert f_{\epsilon} - f\Vert_{L^p(Y_{\epsilon})} + \Vert f_{\epsilon} \Vert_{L^p(Y\setminus Y_{\epsilon})}  + \underbrace{\Vert f\Vert_{L^p(Y\setminus Y_{\epsilon})}}_{\overset{\epsilon\to 0}{\longrightarrow} 0}.
\end{align*}
Since for every $x \in Y_{\epsilon}$ it holds that $supp(\phi_{\epsilon}) \cap (x + Y) \subset B_{\epsilon}(0)$, we obtain using standard arguments for the convolution
%
\begin{align*}
\Vert f_{\epsilon} - f \Vert_{L^p(Y_{\epsilon})}^p 
&=  \int_{Y_{\epsilon}} \left\vert \int_{x - Y} \phi_{\epsilon} (y) (f (x-y) - f(x)) dy \right\vert^p dx
\\
&\le  \int_{Y_{\epsilon}} \left\vert \int_{B_{\epsilon}(0)}  \phi_{\epsilon} (y) (f (x-y) - f(x)) dy \right\vert^p dx
\\
&\le \sup_{y \in B_{\epsilon}(0)} \Vert f(\cdot - y) - f\Vert_{L^p(Y_{\epsilon})}^p
\overset{\epsilon \to 0}{\longrightarrow} 0,
\end{align*}
where the convergence holds due to the Kolmogorov-compactness-criterion. 
Using the fact that for every $y \in Y_{2\epsilon}$ it holds that $\phieps(x) = 0$ for all $x   \in (Y\setminus Y_{\epsilon}) + y$, we obtain
\begin{align*}
 \Vert f_{\epsilon} \Vert_{L^p(Y\setminus Y_{\epsilon})}^p 
&\le \int_Y \vert f(y)\vert^p \int_{(Y\setminus Y_{\epsilon}) + y} \phi_{\epsilon}(x) dx dy
\le \Vert f\Vert_{L^p(Y\setminus Y_{2\epsilon})}^p \overset{\epsilon\to 0}{\longrightarrow} 0.
\end{align*}
This implies the desired result.
\end{proof}
Next, we show that for periodic functions the convolution commutes with the weak derivative. Although the proof is quite simple, this observation is the crucial point in the choice of the definition  of periodic functions in $H^p(\div,Y)$ and the convolution with periodic kernel. 
\begin{proposition}\label{PropCommutingConvolutionDivergence}
Let $f \in H_{\#}^p(\div,Y)$. Then we have 
\begin{align*}
\nabla \cdot f_{\epsilon} = (\nabla \cdot f)_{\epsilon}.
\end{align*}
Especially, for $\nabla \cdot f = 0$ we get $\nabla \cdot f_{\epsilon} = 0$.
\end{proposition}
\begin{proof}
For $x \in Y$ it holds with the divergence theorem
\begin{align*}
\nabla \cdot f_{\epsilon}(x) &= \int_Y \nabla_x \phi_{\epsilon}(x-y) \cdot f (y) dy = - \int_Y \nabla_y \left(\phi_{\epsilon}(x-y)\right) \cdot f(y) dy
\\
&= \int_Y \phi_{\epsilon}(x-y) \nabla_y \cdot f(y) dy - \underbrace{\langle f\cdot \nu , \phi_{\epsilon}(x-\cdot)\rangle_{\partial Y}}_{=0}
= (\nabla \cdot f)_{\epsilon}(x),
\end{align*}
where we used that $\phi_{\epsilon}(x - \cdot) \in W_{\#}^{1,p'}(Y)$.
\end{proof}
\begin{remark}
In the same way we obtain for $f \in W^{1,p}_{\#}(Y)$ that  $\nabla f_{\epsilon} = (\nabla f)_{\epsilon}$.
However, this result is not valid for non-periodic functions.
\end{remark}

\begin{proof}[Proof of Theorem \ref{TheoremDensity}]
Let $f \in H_{\#}^p(\div,Y)$. Since $\phi_{\epsilon}$ is $Y$-periodic, we have that $f_{\epsilon} \in C_{\#}^{\infty}(Y)^n$. Further, from Proposition \ref{PropCommutingConvolutionDivergence} and Lemma \ref{LemmaApproximation} we obtain
\begin{align*}
\Vert f_{\epsilon} - f \Vert_{H^p(\div,Y)} &\le C \left(\Vert f_{\epsilon} - f\Vert_{L^p(Y)} + \Vert \nabla \cdot f_{\epsilon} - \nabla \cdot f \Vert_{L^p(Y)} \right)
\\
&=  C \left(\Vert f_{\epsilon} - f\Vert_{L^p(Y)} + \Vert (\nabla \cdot f)_{\epsilon} - \nabla \cdot f \Vert_{L^p(Y)} \right)
\overset{\epsilon\to 0}{\longrightarrow} 0.
\end{align*}
In a similar way we obtain the density of $C_{\#,\sigma}(Y)^n$ in $L^p_{\#,\sigma}(Y)$, since Proposition \ref{PropCommutingConvolutionDivergence} implies $\nabla \cdot f_{\epsilon} = 0$ for $\nabla \cdot f = 0$.
\end{proof}

\subsection{The cylinder $Z$ with zero normal-trace on $S^{\pm}$}
Now we consider the cell $Z$ and periodic functions with respect to the lateral boundary and with zero-normal trace on $S^{\pm}$. More precisely we define $Z:= Y \times (-1,1):= \left(-\frac12 , \frac12\right)^{n-1}\times (-1,1)$ and the top/bottom $S^{\pm} := Y\times \{\pm 1\}$. For arbitrary $\rho \in (0,\infty]$ we define
\begin{align*}
Z^{\rho} := \left(-\frac12,\frac12\right)^{n-1}\times (-\rho,\rho).
\end{align*}
Further we define  ($W_{\#}^{1,p'}(Z)$ is the space of $Y$-periodic Sobolev functions)
\begin{align}\label{DefinitionHper0Div}
H_{\#,0}^p(\div,Z):= \left\{ u \in H^p(\div,Z)\, : \, \langle u\cdot \nu , \phi \rangle_{\partial Z} = 0 \, \forall \phi \in W^{1,p'}_{\#}(Z)\right\}.
\end{align}
Heuristically spoken this is the space of $Y$-periodic functions with zero normal trace on $S^{\pm}$. Now we define the smooth subspace
\begin{align*}
C_{\#,0}^{\infty}(Z)^n := \left\{ u \in  C_{\#}^{\infty}(Z)^n\, : \, supp(u) \subset \overline{Z}\setminus S^{\pm}\right\},
\end{align*}
and the solenoidal subspaces
\begin{align*}
C_{\#,0,\sigma}^{\infty}(Z)^n &:= \left\{ u \in C_{\#,0}^{\infty}(Z)^n \, : \, \nabla \cdot u = 0 \right\},
\\
L_{\#,\sigma}^p(Z)&:= \left\{ u \in H_{\#,0}^p(\div,Z)\, : \, \nabla \cdot u = 0\right\}.
\end{align*}
The aim of this section is to prove the following Theorem
\begin{theorem}\label{TheoremDensitySmoothCompactPeriodicFunctionsHdiv}
We have that 
\begin{enumerate}
[label = (\roman*)]
\item $C_{\#,0}^{\infty}(Z)^n$ is dense in $H_{\#,0}^p(\div,Z) $,
\item $C_{\#,0,\sigma}^{\infty}(Z)^n$ is dense in $L^p_{\#,\sigma}(Z)$.
\end{enumerate}
\end{theorem}
Again we consider the Dirac-sequence $\phi_{\epsilon}$ defined in $\eqref{DiracSequence}$, but now we extend it  $Y$-periodically to the domain $\R^{n-1} \times (-1,1)$ (and then by zero to the whole $\R^n$). 
 We define the convolution of $f\in L^p(Z)$ by ($x\in Z$)
\begin{align*}
f_{\epsilon}(x):= \int_Z \phi_{\epsilon}(x-y) f(y) dy.
\end{align*}
\begin{lemma}\label{LemmaApproxConvZ}
For $f\in L^p(Z)$ it holds that $f_{\epsilon} \rightarrow f $ in $L^p(Z)$ for $\epsilon \to 0$.
\end{lemma}
\begin{proof}
The proof follows the same lines as the proof of Lemma \ref{LemmaApproximation}. 
\end{proof}
\begin{lemma}\label{LemmaRegularityConvZ}
For $f \in H_{\#,0}^p(\div,Z)$ it holds that
\begin{align*}
\nabla \cdot f_{\epsilon} = (\nabla \cdot f)_{\epsilon}.
\end{align*}
Especially, for $\nabla \cdot f = 0$ we get $\nabla \cdot f_{\epsilon} = 0$.
\end{lemma}
\begin{proof}
We argue in the same way as in the proof of Proposition \ref{PropCommutingConvolutionDivergence} to get for $x \in Z$, using $\phi_{\epsilon}(x-\cdot) \in W^{1,p'}_{\#}(Z)$,
\begin{align*}
\nabla \cdot f_{\epsilon} (x) 
&= \int_Z \phi_{\epsilon} (x-y) \nabla \cdot f(y) dy - \underbrace{\langle f\cdot \nu , \phi_{\epsilon}(x-\cdot)\rangle_{\partial Z}}_{=0}
= (\nabla \cdot f)_{\epsilon}(x).
\end{align*}
\end{proof}
In general for $f \in H_{\#,0}^p(\div,Z)$ the convolution $f_{\epsilon}$ has not a compact support in $\overline{Z}\setminus S^{\pm}$. Hence, we need an additional argument for the construction of a dense sequence in $C_{\#,0}^{\infty}(Z)^n$ for $f$. This procedure is well-known for functions with zero normal trace on the whole boundary, see \cite[Theorem III.2.4]{Galdi}. In the following we extend this method to our setting, where we only point out the new aspects. First of all, we notice that we can extend the function $f$ by zero to the whole strip $Z^{\infty}$ without loosing regularity. More precisely, we have:
\begin{lemma}\label{LemmaZeroExtensionHdiv}
Let $f \in H_{\#,0}^p(\div,Z)$ and we denote by $\tf$ its zero extension to $Z^{\infty}$. Then for every $\rho \gr 1$ it holds that $\tf \in H_{\#,0}^p(\div,Z^{\rho})$ with $\nabla \cdot \tf = \widetilde{\nabla \cdot f}$. Especially, for $\nabla \cdot f = 0$ we get $\nabla \cdot \tf = 0$.
\end{lemma}
\begin{proof}
Let $\phi \in W^{1,p'}_{\#}(Z^{\rho})$ (especially the following computation is valid for all $\phi \in C_0^{\infty}(Z^{\rho})$). It holds with the divergence theorem
\begin{align*}
\int_{Z^{\rho}} \tf \cdot \nabla \phi dy = \int_Z f \cdot \nabla \phi dy = -\int_Z \nabla \cdot f \phi dy + \underbrace{\langle f\cdot \nu , \phi \rangle_{\partial Z}}_{=0} = - \int_{Z^{\rho}} \widetilde{\nabla \cdot f} \phi dy. 
\end{align*}
\end{proof}
Now, let $\rho \gr 1$ and define for $f \in H_{\#,0}^p(\div,Z)$ the function ($y \in Z$)
\begin{align*}
\tf^{\rho}(y) := \left(\tf_1 ,\ldots,\tf_{n-1}, \frac{1}{\rho} \tf_n\right) (\y,\rho y_n).
\end{align*}
\begin{lemma}\label{Lemmafrho}
For $f \in H_{\#,0}^p(\div,Z)$ and $\rho \gr 1$ we have $\tf^{\rho} \in H_{\#,0}^p(\div,Z)$ with $\nabla \cdot \tf^{\rho} = \nabla \cdot \tf(\y,\rho y_n)$, and $\tf^{\rho}$ has compact support in $\overline{Z}\setminus S^{\pm}$. Further, it holds that
\begin{align*}
\tf^{\rho} \overset{\rho \searrow 1}{\longrightarrow} \tf = f \qquad\mbox{ in } H^p(\div,Z).
\end{align*}
Especially, for $\nabla \cdot f= 0$ it holds that $\nabla \cdot \tf^{\rho} = 0$.
\end{lemma}
\begin{proof}
The compact support of $\tf^{\rho}$ is obvious. Let us show the regularity. For $\phi \in W_{\#}^{1,p'}(Z)$ it holds  with Lemma \ref{LemmaZeroExtensionHdiv} that (we use the notation $\phi^{\frac{1}{\rho}}(y):= \phi\left(\y,\frac{y_n}{\rho}\right) \in W_{\#}^{1,p'}(Z^{\rho})$)
\begin{align*}
\int_Z \nabla \phi \cdot \tf^{\rho} dy 
= \frac{1}{\rho} \int_{Z^{\rho}} \nabla \phi^{\frac{1}{\rho}} \cdot \tf dy
&= -\frac{1}{\rho} \int_{Z^{\rho}} \phi^{\frac{1}{\rho}} \nabla \cdot \tf dy + \frac{1}{\rho} \underbrace{ \langle \tilde{f}\cdot \nu ,\phi^{\frac{1}{\rho}} \rangle_{\partial Z^{\rho}}}_{=0}
\\
&= - \int_Z \phi \nabla \cdot \tf (\y,\rho y_n) dy.
\end{align*}
The convergence of $\tf^{\rho}$ to $f$ in $H^p(\div,Z)$ is quite standard and follows from the Kolmogorov-compactness criterion, see for example \cite[Section 1.1]{WlokaEnglisch}
\end{proof}
Now we can give the proof of Theorem \ref{TheoremDensitySmoothCompactPeriodicFunctionsHdiv}:
\begin{proof}[Proof of Theorem \ref{TheoremDensitySmoothCompactPeriodicFunctionsHdiv}]
Let $f \in H_{\#,0}^p(\div,Z)$ and $\eta \gr 0$.  From Lemma \ref{Lemmafrho} we obtain the existence of $\rho \gr 1$ such that
\begin{align*}
\Vert f - \tf^{\rho}\Vert_{H^p(\div,Z)} \le \frac{\eta}{2}.
\end{align*}
We define 
\begin{align*}
\tf^{\rho}_{\epsilon}:= (\tf^{\rho})_{\epsilon} \in C^{\infty}(Z)^n.
\end{align*}
From the $Y$-periodicity of $\phi_{\epsilon}$ we get the $Y$-periodicity of $\tf^{\rho}_{\epsilon}$. Further, since $\tf^{\rho}$ has a compact support in $\overline{Z}\setminus S^{\pm}$, $\tf^{\rho}_{\epsilon}$ has a compact support in $\overline{Z}\setminus S^{\pm}$ for $\epsilon $ small enough (for fixed $\rho$). Hence, $\tf^{\rho}_{\epsilon} \in C_{\#,0}^{\infty}(Z)^n$. Now Lemma \ref{LemmaApproxConvZ} and \ref{LemmaRegularityConvZ} imply that for $\epsilon $ small enough it holds that
\begin{align*}
\Vert \tf^{\rho}_{\epsilon} - \tf^{\rho} \Vert_{H^p(\div,Z)} \le \frac{\eta}{2}.
\end{align*} 
Together we obtain
\begin{align*}
\Vert f - \tf^{\rho}_{\epsilon}\Vert_{H^p(\div,Z)} \le \eta.
\end{align*}
By similar arguments we obtain the density of $C_{\#,0,\sigma}^{\infty}(Z)^n$ in $L^p_{\#,\sigma}(Z)$.
\end{proof}

\end{appendix}

\section*{Acknowledgement}
The first author is  supported by the project SCIDATOS (Scientific Computing for Improved Detection and Therapy of Sepsis), which was funded by the Klaus Tschira Foundation, Germany (Grant Number 00.0277.2015).

\bibliographystyle{abbrv}
\bibliography{literature}

\end{document}